\documentclass[11pt]{amsart}

\usepackage{amsmath,amsthm,amsfonts,amssymb,mathrsfs}
\usepackage{inputenc,mathrsfs}

\numberwithin{equation}{section}
\usepackage[dvips]{graphicx}
\usepackage[colorlinks]{hyperref}

\newtheorem{theorem}{Theorem}[section]
\newtheorem{lemma}[theorem]{Lemma}
\newtheorem{corollary}[theorem]{Corollary}
\newtheorem{proposition}[theorem]{Proposition}
\newtheorem{definition}[theorem]{Definition}
\theoremstyle{remark}
\newtheorem{remark}{Remark}[section]

\newcommand{\tr}{\mathrm{tr}}
\newcommand{\sym}{\mathrm{sym}}

\newcommand{\riem}{\mathrm{Rm}}
\newcommand{\ric}{\mathrm{Ric}}
\newcommand{\scal}{\mathrm{Scal} }
\newcommand{\Div}{\mathrm{div}}
\newcommand{\curl}{\mathrm{curl}}
\newcommand{\oct}{\circledcirc}
\newcommand{\Pop}{\mathsf P}
\newcommand{\Vop}{\mathsf V}
\newcommand{\Lop}{\mathsf L}
\newcommand{\vol}{\mathrm{Vol}}

\newcommand{\Fop}{\mathsf{F}}
\newcommand{\Kop}{\mathsf{K}}

\newcounter{noteCounter}
\title{A $G_2$-Hilbert functional in $G_2$-geometry}

\author{Panagiotis Gianniotis}
\address{Panagiotis Gianniotis: Department of Mathematics\\
National and Kapodistrian University of Athens}
\email{pgianniotis@math.uoa.gr}

\author{George Zacharopoulos}
\address{George Zacharopoulos: Department of Mathematics\\
National and Kapodistrian University of Athens}
\email{gzacharop@math.uoa.gr}

\begin{document}
\maketitle

\begin{abstract}
In this paper we introduce a new functional on the space of $G_2$-structures which we call the $G_2$--Hilbert functional. It is uniquely determined by a few basic principles inspired by the Einstein--Hilbert functional in Riemannian Geometry, and it has similar variational behaviour with it. For instance,  torsion-free and nearly $G_2$-structures are saddle critical points of the volume-normalized $G_2$--Hilbert functional. This allows us to uniquely distinguish two  new  flows of $G_2$-structures, which can be considered as analogues of the Ricci flow in $G_2$-geometry.
\end{abstract}

\section{Introduction} $\\$

Geometric flows of $G_2$-structures have recently drawn a considerable amount of attention as a possible tool to deform a given $G_2$-structure to one that has special properties.

The most well studied geometric flow on $G_2$-structures is the Laplacian flow, introduced by Bryant in \cite{Bryant}. This is a flow of \textit{closed} $G_2$-structures, evolving in the direction of their Hodge Laplacian. Namely, a family $\{\varphi_t\}$ of closed $G_2$-structures evolves by the Laplacian flow if
\begin{equation}
\begin{aligned}
\frac{\partial \varphi}{\partial t} &=\Delta_d \varphi.
\end{aligned}
\end{equation}
The short-time existence of smooth solutions of the Laplacian flow starting from any given closed $G_2$-structure on a compact $7$-manifold was proven by Bryant-Xu in \cite{Bryant_Xu}. In fact, under the Laplacian flow the closed $G_2$-structures evolve within   the same cohomology class and the fixed points of the flow are exactly the torsion-free $G_2$-structures in that cohomology class, if they exist. Moreover, by Hitchin \cite{Hitchin}, the Laplacian is the gradient flow of the volume functional restricted to a given cohomology class.

One can think of the Laplacian flow as a $G_2$-analogue of the K\"ahler--Ricci flow, in which the K\"ahler form evolves within its cohomology class so that the underlying Riemannian metric evolves by the Ricci flow. A notable difference with the Laplacian flow, however,  is that the Ricci flow exists for short-time starting from any initial Riemannian metric, and becomes a K\"ahler--Ricci flow once the initial condition is K\"ahler.  The Laplacian flow on the other hand does not define a (weakly) parabolic flow of general (i.e. non-closed) $G_2$-structures, and solutions starting from a general initial $G_2$-structure are not expected to exist even for short time. Moreover, it is not known whether every closed $7$-manifold admitting $G_2$-structures admits a closed $G_2$-structure. 

On the other hand, co-closed $G_2$-structures are abundant, since every compact $7$-manifold admitting $G_2$-structures admits a co-closed $G_2$-structure, by Crowley-Nordstr\"om \cite{CJ}. Karigiannis--McKay--Tsui in \cite{coflow} introduced the Laplacian co-flow, which evolves a family of co-closed $G_2$-structures so that 
$$\frac{\partial \psi}{\partial t} = \Delta_d \psi$$
where $\psi$ denotes the dual of a $G_2$-structure $\varphi$. 
To this day, however, it has not been possible to establish the short-time existence of this flow, starting from a co-closed $G_2$-structure. To amend this issue Grigorian \cite{mod_coflow} proposed a modified Laplacian co-flow, and established its short time existence.

The Laplacian flow, the Laplacian co-flow and the modified Laplacian co-flow have been extensively studied during the past years, especially under additional symmetry assumptions, see for instance \cite{LW1,LW2,Laplacian_stability,Lauret1,Lauret2,FR,Lin,Lambert_Lotay,Kennon_Lotay,MS_coflow,SSS_coflow,Grigorian_mcoflow_sym,Grigorian_coflow_survey} for a non-exhaustive list. In particular, the dynamic stability of the Laplacian flow and modified co-flow has been established, near torsion-free $G_2$-structures, by Lotay--Wei \cite{Laplacian_stability} and Bedulli--Vezzoni \cite{mod_coflow_stability} respectively.

 The general underlying structure, however, of these flows is still poorly understood. For instance,  there are no known preserved geometric conditions and crucial analytic tools to perform singularity analysis, such a result analogous to Perelman's no-local collapsing result for the Ricci flow, remain elusive - even for the Laplacian flow. Gao Chen \cite{Gao} proves such a result, however under the assumption that the norm of the torsion of the evolving $G_2$-structures remains bounded along the Laplacian flow.

In seeking a reasonable geometric evolution equation it is typical to consider the negative gradient flow of some natural diffeomorphism invariant functional.  In $G_2$-geometry, a possible such functional is given by the square $L^2$-norm of the torsion of a $G_2$-structure, namely
\begin{equation}\label{eqn:T2_functional}
\varphi \mapsto \frac{1}{2}\int_M |T|^2 d\mu,
\end{equation}
where $T$ is the torsion of $\varphi$, while the norm and the volume form are defined by the Riemannian metric induced by $\varphi$. Purely by scaling considerations, critical points of this functional are exactly the torsion-free $G_2$-structures. Weiss--Witt in \cite{WW} established the short-time existence of  the negative gradient flow of an equivalent version of \eqref{eqn:T2_functional} and proved its dynamic stability near torsion-free $G_2$-structures. Moreover, Ammann--Weiss--Witt in \cite{spinorial}, produce a general spinorial flow, the negative gradient flow of an energy functional on spinors, which is closely related to the flow introduced in \cite{WW}.

It is also possible to restrict the functional \eqref{eqn:T2_functional} in the isometric class of a given $G_2$-structure $\bar \varphi$, namely the collection of all $G_2$-structures that induce the same Riemannian metric as $\bar \varphi$. These are sections of an $\mathbb RP^7$-bundle oven the $7$-manifold $M$, by \cite{Bryant}. The negative gradient flow of this restriction is then given by
\begin{equation}\label{eqn:isometric_flow}
\frac{\partial \varphi}{\partial t} = \Div T\lrcorner \psi.
\end{equation}
It was first introduced by Grigorian in \cite{Grig_octonion}, and its short-time existence on compact $7$-manifolds was established by Bagaglini in \cite{Bagaglini}. Several properties of this flow were further investigated by Grigorian \cite{Grig_isometric}, as well as from the first named author together with Dwivedi and Karigiannis in \cite{DGK}. Loubeau-S\'a Earp showed that this flow is a particular case of a harmonic flow of geometric structures \cite{harmonic_geometric}, and shares many similarities with the harmonic map heat flow. It also generalizes to the case of $\textrm{Spin}(7)$-structures, see \cite{harmonic_spin7}.

In \cite{flows2}, following earlier work of Karigiannis \cite{flows1}, the authors classified all the possible second-order quasilinear operators that are given by the covariant derivative of the torsion and define a $3$-form, so they can define a flow of $G_2$-structures. The analysis in \cite{flows2} distinguished a wide class of differential operators on $G_2$-structures, that define flows amenable to the DeTurck trick to deal with the diffeomorphism invariance of the equation, and prove that these flows have unique smooth solutions, at least for short time, starting from any initial $G_2$-structure on a compact $7$-manifold. The most interesting case of this result is the following theorem, which is included in Theorem 6.76 in \cite{flows2}, and includes the negative gradient flow of the functional \eqref{eqn:T2_functional}.

\begin{theorem}\label{thm:general_ste}
Let $M$ be a compact $7$-manifold with a $G_2$-structure $\varphi_0$. Consider the initial value problem 
\begin{equation}\label{eqn:general_flow}
\begin{aligned}
\frac{\partial \varphi}{\partial t}&= (-\ric +a\mathcal L_{\Vop T} g)\diamond \varphi + (b_1 \Div T +b_2 \Div T^t) \lrcorner \psi +\textrm{lower order terms},\\
\varphi(0)&=\varphi_0,
\end{aligned}
\end{equation}
where $\ric$ is the Ricci curvature of the Riemannian metrics induced by the evolving $G_2$-structures, $T$ is their torsion and $T^t$ its transpose. Using abstract index notation, $\Vop T_k=T_{ab}\varphi_{abk}$, and $\diamond \varphi$ acts on symmetric $2$-tensors as
$$(h\diamond\varphi)_{ijk} =h_{ip}\varphi_{pjk}+h_{jp} \varphi_{ipk} + h_{kp}\varphi_{ijp}.$$
Then, if
\begin{align*}
0\leq b_1-a-1<4,\\
b_1+b_2\geq 1,
\end{align*}
there exists an $\varepsilon=\varepsilon(\varphi_0)>0$ and a unique smooth one parameter family $\varphi(t)$, $t\in[0,\varepsilon)$, of $G_2$-structures that solves \eqref{eqn:general_flow}.
\end{theorem}

The particular case where $b_1+b_2=1$, $a=-b_2$, corresponds to the $1$-parameter family of second order quasilinear operators
\begin{equation}\label{eqn:special_rl}
P_a(\varphi) = (-\ric+a\mathcal L_{\Vop T} g)\diamond \varphi +( (1+a) \Div T - a\nabla \tr T).
\end{equation}
In this paper, we call any operator  on $G_2$-structures which, neglecting lower order terms, is of the form \eqref{eqn:special_rl},  a \textit{special Ricci-like operator}.  Note that we replaced the operator $\Div T^t$ with $\nabla \tr T$ as these two operators are related via lower order terms, see Proposition \ref{prop:g2bianchi_identities}.

Our interest on special Ricci-like operators stems from the fact that the behaviour of the principal symbol of the linearization of the operators $P_a$ is strikingly similar to that of the linearization of the Ricci tensor in Riemannian geometry. In both cases, the principal symbol has a $7$-dimensional kernel, due to diffeomorphism invariance, and there is an invariant subspace in which it acts as a multiple of the identity. The reader can find in \cite{chow_knopf} a detailed analysis of the principal symbol of the Ricci tensor, while the precise behaviour of the principal symbol of special Ricci-like operators is described in Proposition \ref{prop:special_rl_symbol}.

By \eqref{eqn:metric_evol}  the geometric flow
\begin{equation}\label{eqn:srl_flow_intro}
\frac{\partial \varphi}{\partial t} = P_a(\varphi)
\end{equation}
induces an evolution of Riemannian metrics $g(t)$ which is just the Ricci flow modified by diffeomorphisms, namely
\begin{equation}\label{eqn:srl_metric_flow}
\frac{\partial g}{\partial t} = -2\ric +2a\mathcal L_{\Vop T} g.
\end{equation}
It is thus reasonable to consider geometric flows of $G_2$-structures induced by special Ricci-like operators. Moreover, the behaviour of the principal symbol of these operators raises the expectation that, under a flow defined via a special Ricci-like operator, geometric quantities will evolve via the heat equation, up to lower order terms. This is explained in more detail in Subsection \ref{sec:diff_ops}. It is important to note that this is a drawback of the negative gradient-flow of \eqref{eqn:T2_functional}. This flow, by Proposition \ref{variationofbasicfunctionals}, is given, up to lower-order terms, by
\begin{equation*}
\frac{\partial \varphi}{\partial t} = (-\ric-\frac{1}{2}\mathcal L_{\Vop T} g +\textrm{lower order terms})\diamond \varphi + \Div T\lrcorner \psi.
\end{equation*}
In particular, the second-order terms are not described by a special Ricci-like operator, which we will see in Subsection \ref{sec:diff_ops} to be crucial for the torsion to evolve by a heat equation.

Since $G_2$-structures, unlike Riemannian metrics, do have first order invariants, expressed in terms of their torsion $T$, the definition of the flow \eqref{eqn:srl_flow_intro} is arbitrary, as it doesn't include any lower order terms. The choice of these lower-order terms should be chosen so that they simplify the structure of the evolution equations of various geometric quantities, or at least be dictated by some reasonable principles. 

The purpose of this paper is to identify principles that lead to a judicious choice of these lower-order terms. These principles do have variational origin, however the flows of $G_2$-structures they will specify will not be the negative gradient flow of some functional on $G_2$-structures, in the same way that the Ricci flow is not the gradient flow of a functional on Riemannian metrics.

In order to describe our results, recall that the subspace on which the principal symbol of the linearization of Ricci curvature acts as a multiple of the identity is exactly the kernel of the principal symbol of the Bianchi operator 
 $$\beta_g(h)= \Div\left(h-\frac{\tr h}{2} g\right).$$
Moreover, the twice contracted second Bianchi identity 
\begin{equation}\label{eqn:twice_bianchi}
\Div \ric = \frac{1}{2}\nabla \scal,
\end{equation}
is just the statement that $\ric\in\ker \beta_g$, and has a variational origin; it is a consequence of the diffeomorphism invariance of the Einstein--Hilbert (or total scalar curvature) functional
$$\mathcal S(g)=\int_M \scal d\mu_g$$
on the space of smooth Riemannian metrics on $M$. Moreover, the Ricci tensor, although it is not the gradient of $\mathcal S$, it also has variational origin in the sense that the first variation of $\mathcal S$ satisfies
\begin{equation}\label{eqn:EH_var}
D_g \mathcal S(h)= \int_M \langle -\ric  +\frac{1}{2}\scal \; g,h\rangle d\mu_g
\end{equation}
for any infinitesimal variation $h$ of $g$, see for instance \cite{besse}. In particular, the appearance of the Ricci curvature in \eqref{eqn:EH_var} is due to the variation formula
\begin{equation}\label{eqn:var_scal}
D_g \scal (h) = \langle -\ric,h\rangle -\Delta \tr h +\Div\Div h.
\end{equation}

In this work, our goal is to distinguish special Ricci-like operators that have a similar variational origin as the Ricci tensor in Riemannian geometry.

In Proposition \ref{prop:uniqueness}, we show that  there is no diffeomorphism invariant functional on $G_2$-structures whose gradient is a special Ricci-like operator, in the same way that the Ricci tensor is not the gradient of $\mathcal S$. Moreover, we show that if the gradient of a diffeomorphism invariant functional is a second order quasilinear operator on $G_2$-structures and \textit{includes} a special Ricci-like operator, then the second order terms of the gradient are exactly given by
\begin{equation}\label{eqn:frl}
P_{-1/3} +\left(\frac{1}{6}\scal\; g\right)\diamond \varphi,
\end{equation}
which is not a special Ricci-like operator. 

In Section \ref{sec:g2_hilbert}, under reasonable assumptions, we show that there is a \textit{unique} functional on the space of $G_2$-structures whose gradient has exactly the second order terms \eqref{eqn:frl}. We call this functional the $G_2$-Hilbert functional, and is given by 
\begin{equation*}
\mathcal F(\varphi)=\int_M \left(\frac{1}{6}\scal -\frac{1}{3} |T|^2 -\frac{1}{6}(\tr T)^2\right) d\mu_g.
\end{equation*}
The critical points of its volume-normalized version include the torsion free and nearly $G_2$-structures, namely $G_2$-structures whose torsion satisfies $T=cg$, for some constant $c$. Moreover, in Section \ref{sec:2var}, we show that torsion free and nearly-$G_2$ critical points are saddle points of the volume normalized $G_2$-Hilbert functional, a behaviour which is exactly analogous to that of the volume normalized Einstein--Hilbert functional, where Einstein metrics are always saddle critical points.

In Section \ref{sec:g2_hilbert}, we also show that the variational behaviour of $\mathcal F$ is naturally linked to two particular special Ricci-like operators
\begin{align*}
\hat P(\varphi) &= \hat P_1(\varphi) \diamond \varphi + P_2(\varphi)\lrcorner \psi,\\
\tilde P(\varphi) &= \tilde P_1(\varphi)\diamond\varphi + P_2(\varphi)\lrcorner \psi,
\end{align*} 
whose second-order terms are given by $P_{-1/3}$. In particular,
\begin{align*}
\hat P_1 &= -\ric -\frac{1}{3}\mathcal L_{\Vop T} g-\frac{2}{3} (T\circ (\Vop T\lrcorner \varphi))_\sym + \tr T T_\sym,\\
\tilde P_1 &= \hat P_1 +\frac{1}{3}\left(|T|^2 -\frac{1}{3} |\Vop T|^2\right) g,\\
P_2 &= \frac{2}{3} \Div T +\frac{1}{3} \nabla \tr T +\frac{1}{3} \tr T \Vop T,
\end{align*}
where we write $(A\circ B)_{ij} = A_{ip}B_{pj}$, for any $2$-tensors $A,B$, and $A_\sym = \frac{1}{2} (A+A^t)$ denotes the symmetric part of $A$.

 These operators are uniquely determined by the following properties:
\begin{enumerate}
\item The first variation of the quantity 
$$\Fop(\varphi)=\frac{1}{6}\scal -\frac{1}{3} |T|^2 -\frac{1}{6}(\tr T)^2$$
is given by 
$$D_\varphi \Fop (h\diamond\varphi+X\lrcorner\psi) = \langle \hat P_1(\varphi),h\rangle +\langle P_2(\varphi),X\rangle + \textrm{divergence terms},$$
which is analogous to the variation formula \eqref{eqn:var_scal}. In fact, the quantity $\Fop$ is the unique quantity,  linear in the scalar curvature and quadratic in the torsion, and $\hat P$ the unique special Ricci-like operator that satisfy this property.
\item There is a first order, Bianchi-like, operator $\tilde B$, acting on any pair $(h,X)$ of a symmetric $2$-tensor $h$ and $1$-form $X$, such that $\tilde P$ satisfies an analogue of the twice contracted second Bianchi identity \eqref{eqn:twice_bianchi}, in the sense that $(\tilde P_1(\varphi),P_2(\varphi))\in \ker \tilde B$. Moreover, the principal symbol of the linearizations of $P_{-1/3}$ acts on the kernel of the principal symbol of $\tilde B$ as a multiple of the identity. 
\end{enumerate}

The corresponding flows of $G_2$-structures
\begin{align*}
\frac{\partial \varphi}{\partial t} &= \hat P(\varphi),\\
\frac{\partial \varphi}{\partial t} &= \tilde P(\varphi),
\end{align*}
are both included in Theorem \ref{eqn:general_flow}, and can be considered as analogues of the Ricci flow in $G_2$-geometry. Interestingly, although Einstein metrics with positive scalar curvature are homothetically shrinking along Ricci flow, nearly $G_2$ structures, which are Einstein metrics with positive scalar curvature, are homothetically expanding along both these flows. See Subsection \ref{sec:flows} for more details.

Closely related to our work is the work of Streets--Tian \cite{ST} in the setting of Hermitian geometry. In \cite{ST}, they define the Hermitian--Hilbert functional and show that it is the unique functional which leads to a flow of Hermitian metrics driven, up to lower order terms, by the Chern Ricci curvature, which they call the Hermitian curvature flow. In particular, Chern Ricci curvature is a strictly elliptic differential operator; the principal symbol of its linearization has no kernel since there is no diffeomorphism invariance. From the point of view of this paper, it plays a role analogous to special Ricci-like operators. However, it is not clear whether any of the operators $\hat P$ or $\tilde P$ we distinguish in this paper have a similar geometric origin, namely  whether they appear as the Ricci curvature of a special connection induced on the tangent bundle by a $G_2$-structure. Moreover, in contrast to \cite{ST}, critical points of the $G_2$-Hilbert functional - at least those that are torsion free or nearly $G_2$, are saddle points of the functional.

The Hermitian curvature flow reduces to the K\"ahler Ricci flow, if the initial K\"ahler form is closed. This naturally raises the question whether the flows proposed in this paper reduce to the Laplacian flow, if we start with a closed $G_2$-structure. It is important to note that in the Hermitian setting,
vanishing of the torsion of the Chern connection is equivalent to the underlying metric being K\"ahler, so the closed K\"ahler form may still admit a non-trivial evolution by K\"ahler Ricci flow. However, in the $G_2$-setting, the vanishing of the torsion means that the $G_2$-structure is a fixed point of all the flows described above, as well as of the Laplacian flow, while the $G_2$-structure being closed does not imply the vanishing of the torsion.

\large{\bf Acknowledgments:}  The authors would like to thank Shubham Dwivedi and Spiro Karigiannis for many interesting discussions. This research was supported by the Hellenic Foundation for Research
and Innovation (H.F.R.I.) under the “2nd Call for H.F.R.I. Research Projects to support
Faculty Members \& Researchers” (Project Number: HFRI-FM20-2958).

\section{Preliminaries in $G_2$-geometry}
In this section we recall some foundational material on $G_2$ structures on a $7$-manifold, following \cite{flows1,flows2}. In particular, we regard differential forms as skew-symmetric tensors, so given a Riemannian metric $g$ on a smooth $7$-manifold $M$, we will use the convention that given $\alpha,\beta\in \Omega^k$, 
\begin{equation*}
\langle\alpha,\beta\rangle =\alpha_{i_1\cdots i_k} \beta_{i_1\cdots i_k},
\end{equation*}
in abstract index notation.

A $G_2$-structure on a $7$-manifold $M$ is a $3$-form $\varphi$ with the property that for any two vector fields $V,W$ on $M$,
\begin{equation}\label{eqn:B}
B(V,W) =(V \lrcorner \varphi )\wedge (W \lrcorner \varphi) \wedge \varphi \in \Omega^7
\end{equation}
is a non-vanishing $7$-form on $M$. Thus a $G_2$-structure, via $B$, defines an orientation on $M$. Moreover, it defines a Riemannian metric $g_\varphi$ on $M$ such that the relation
$$B(V,W) = -6 g_\varphi(V,W) d\mu_{g_\varphi}$$
holds.

Now, the orientation and Riemannian metric induced by a $G_2$-structure, induce the Hodge operator $*_{\varphi}: \Omega^k \rightarrow \Omega^{7-k}$. We define the dual $4$-from $\psi=*_{\varphi} \varphi$, so that $d\mu_{g_\varphi} = \frac{1}{7} \varphi \wedge \psi$.

We will denote by $\Omega^3_{+}$ the space of all $G_2$-structures on $M$.  Since condition \eqref{eqn:B} is open, it follows that $\Omega^3_{+}$ is an open subset of $\Omega^3$.

\subsection{Contraction identities}
There are several fundamental contraction identities between a $G_2$-structure $\varphi$ and its dual $4$-form $\psi$. Namely,
\begin{equation}\label{eqn:g2identities_1}
\begin{aligned}
\varphi_{ijp} \varphi_{klp} &= g_{ik} g_{jl}-g_{il}g_{jk} - \psi_{ijkl},\\
\varphi_{ipq}\varphi_{jpq}&= 6g_{ij},\\
\varphi_{ijk}\varphi_{ijk}&=42.
\end{aligned}
\end{equation}
\begin{equation}\label{eqn:g2identities_2}
\begin{aligned}
\psi_{ijkl} \psi_{pqkl} = & 4g_{ip}g_{jq} -4 g_{iq}g_{jp} -2 \psi_{ijpq},\\
\psi_{ijkl} \psi_{ajkl} = & 24g_{ia}, \\
\psi_{ijkl} \psi_{ijkl} = & 168.
\end{aligned}
\end{equation}
\begin{equation}\label{eqn:g2identities_3}
\begin{aligned}
\varphi_{ijk}\psi_{abck} =& g_{ia}\varphi_{jbc}+ g_{ib}\varphi_{ajc} + g_{ic}\varphi_{abj}-g_{ja}\varphi_{ibc}-g_{jb}\varphi_{aic} -g_{jc} \varphi_{abi},\\
\varphi_{ijk} \psi_{abjk} =& -4 \varphi_{iab},\\
\varphi_{ijk} \psi_{aijk} = & 0.
\end{aligned}
\end{equation}

\subsection{Decomposition of forms}
A $G_2$-structure $\varphi$ induces the pointwise orthogonal decompositions
\begin{equation}\label{eqn:decompositions}
\begin{aligned}
\Omega^3&=\Omega^3_{1} \oplus \Omega^3_{27} \oplus \Omega^3_7\\
\Omega^2&=\Omega^2_{14}\oplus \Omega^2_7
\end{aligned}
\end{equation}
where $\Omega^k_l$ has pointwise dimension $l$. 

In particular, any $3$-form $\omega= \frac{1}{6}\omega_{ijk} dx^i\wedge dx^j\wedge dx^k$ can be decomposed as 
\begin{equation}\label{eqn:hX_decomposition}
\omega= h\diamond \varphi + X\lrcorner \psi
\end{equation}
where $h\in \mathcal S^2 =\Gamma(S^2(T^*M))$ is a symmetric $2$-tensor and $X$ a $1$-form on $M$, so that in local coordinates 
\begin{align*}
h\diamond \varphi &= \frac{1}{6} (h\diamond \varphi)_{ijk} dx^i \wedge dx^j \wedge dx^k\\
X\lrcorner \psi&=\frac{1}{6} (X\lrcorner \psi)_{ijk}dx^i \wedge dx^j \wedge dx^k,
\end{align*}
where
$$(h\diamond \varphi)_{ijk} = g^{pq}(h_{ip} \varphi_{qjk} + h_{jp}\varphi_{iqk} + h_{kp}\varphi_{ijq}),$$
and 
$$(X\lrcorner \psi)_{ijk} = g^{ml}X_m \psi_{lijk}.$$
Thus, in abstract index notation, 
\begin{align*}
(h\diamond\varphi)_{ijk}&= h_{ip}\varphi_{pjk} + h_{jp} \varphi_{ipk} +h_{kp}\varphi_{ijp},\\
(X\lrcorner\psi)_{ijk} &= X_m \psi_{mijk}.
\end{align*}
The decomposition of $\Omega^3_{1}\oplus\Omega^3_{27}$ in \eqref{eqn:decompositions} corresponds to the trace and trace-free part of $h$ respectively, while $X$ corresponds to the $\Omega^3_7$ summand. In more detail, we can easily that 
$$(fg)\diamond\varphi=3f (g\diamond\varphi)$$
and if $\eta\in \Omega^3_7$ then 
\begin{equation}\label{eqn:eta_to_X}
\begin{aligned}
\eta&=X\lrcorner\psi,\\
X_m&=\frac{1}{24}  \eta_{abc} \psi_{abcm},
\end{aligned}
\end{equation}
which follows by contracting the first equation in \eqref{eqn:eta_to_X} by $\psi$ and applying the corresponding contraction identity.

\begin{remark}\label{rmk:L2_ips}
Decomposition \eqref{eqn:hX_decomposition} defines an isomorphism between the vector bundles $S^2(T^*M)\oplus T^*M$ and $\Lambda^3(M)$, via the assignment
\begin{equation}\label{eqn:isomorphism}
(h,X) \mapsto h\diamond\varphi+X\lrcorner \psi.
\end{equation}
The bundle $\Lambda^3(M)$ is equipped with the natural inner-product arising from $3$-forms being skew-symmetric tensors, while $S^2(T^*M)\oplus T^*M$ comes with the natural inner-product
$$\langle (h,X),(w,Y)\rangle = \langle h,w\rangle+\langle X,Y\rangle.$$
It is important to observe that isomorphism \eqref{eqn:isomorphism} \textit{is not} an isometry, see for instance Proposition 2.27 in \cite{flows2}. To see this, observe that applying the contraction identities \eqref{eqn:g2identities_1}, \eqref{eqn:g2identities_2} and \eqref{eqn:g2identities_3} we obtain that for smooth functions $\alpha,\beta$, symmetric $2$-tensors $h,w$ with 
\begin{align*}
h&= h_0 + \frac{\tr h}{7} g,\\
w&=w_0+\frac{\tr w}{7} g,
\end{align*}
where $h_0,w_0$ are their trace-free parts, and $1$-forms $X,Y$, we have
\begin{equation*}
\begin{aligned}
\langle h\diamond \varphi + X\lrcorner\psi,w\diamond \varphi + Y\lrcorner\psi \rangle&= \frac{54}{7} \tr h \tr w +12\langle h_0,w_0\rangle + 24 \langle X,Y\rangle,\\
&=378 \langle \frac{\tr h}{7} g, \frac{\tr w}{7} g\rangle +12\langle h_0,w_0\rangle + 24 \langle X,Y\rangle.
\end{aligned}
\end{equation*}
\end{remark}

\medskip{}

In order to describe the decomposition $\Omega^2=\Omega^2_{14}\oplus \Omega^2_7$ it is useful to define the operator $\Pop:\Omega^2 \rightarrow \Omega^2$ as $\Pop(\eta)_{ij}=\eta_{ab}\psi_{abij}$, which is self-adjoint since
$$\langle \Pop(\alpha),\beta\rangle = \langle \alpha,\Pop(\beta)\rangle.$$
The decomposition $\Omega^2_{14}\oplus \Omega^2_7$ corresponds to the eigenvalues $2$ and $-4$ of $\Pop$ respectively. In particular $\alpha \in \Omega^2_{14}$ if and only if $\alpha_{ij}\varphi_{ijk}=0$
and every $\alpha\in\Omega^2_7$ is of the form $\alpha_{ij}=X_m\varphi_{mij}$ for $X_m=\frac{1}{6}\alpha_{ij}\varphi_{ijm}$. 

It will be convenient to introduce the operator $\Vop:\mathcal T^2 \rightarrow \Omega^1$, so that for any $2$-tensor $\alpha$ 
$$\Vop(\alpha)_k=\alpha_{ij} \varphi_{ijk}.$$
It is not hard to verify that $\alpha = \frac{1}{6}\Vop(\alpha)\lrcorner \varphi$ for any $\alpha \in \Omega^2_7$.

Finally, we define the operator $\curl$ acting on any $k$-tensor $A$ as follows
$$\curl(A)_{i_1\cdots i_k} = \nabla_a A_{b i_1 \cdots i_{k-1}} \varphi_{abi_k}.$$ 

\subsection{Torsion and the $G_2$-Bianchi identity}
Let $X$ be a vector field on a $7$-manifold $M$ with a $G_2$-structure $\varphi$. Then $\nabla_X \varphi\in\Omega^3_7$, see for instance \cite{flows1}, so there is a $2$-tensor $T$ such that $$\nabla_X \varphi_{ijk} = X_m T_{ml}\psi_{lijk}.$$ We call $T$ the torsion of the $G_2$-structure. In fact, using \eqref{eqn:eta_to_X}, we obtain that
$$T_{pq}=\frac{1}{24} \nabla_p \varphi_{ijk} \psi_{ijkq}.$$
The covariant derivative of $\psi$ can also be expressed in terms of the torsion via the identity
\begin{equation*}
\nabla_p \psi_{ijkl} = -T_{pi} \varphi_{jkl} + T_{pj}\varphi_{ikl} -T_{pk}\varphi_{ijl} + T_{pl}\varphi_{ijk}.
\end{equation*}

The covariant derivative of the torsion satisfies the following important identity, due to Karigiannis \cite{flows1}, and is a consequence of the diffeomorphism invariance of the torsion.
\begin{equation}\label{eqn:g2_bianchi}
\nabla_i T_{jk} - \nabla_j T_{ik} =\left( T_{ia} T_{jb} +\frac{1}{2} R_{ijab}\right) \varphi_{abk}
\end{equation}
We will refer to identity \eqref{eqn:g2_bianchi} as the $G_2$-Bianchi identity.

Taking various contractions and traces of \eqref{eqn:g2_bianchi} we obtain the following identities. 
\begin{proposition}
[Corollary 5.18 in \cite{flows2}]\label{prop:g2bianchi_identities}
\begin{align*}
\scal &= -2\Div \Vop T+(\tr T)^2 - \langle T,T^t\rangle - \langle T,\Pop(T)\rangle, \\
\ric &= - (K_2)_\sym -\frac{1}{2}\mathcal L_{\Vop T} g +\tr T T_\sym - (T^2)_\sym, \\
F&= 4 (K_3)_\sym - 2 (T\oct T)_\sym\\
\Div T^t &=\nabla \tr T+T(\Vop T),  \\
\langle \nabla T,\psi\rangle&= \tr T \Vop T - \Vop (T^2) -T^t(\Vop T),
\end{align*}
where $F_{jk}=R_{abcd}\varphi_{abj}\varphi_{cdk}$ is a symmetric $2$-tensor and
\begin{align*}
(K_2)_{ab} &= \nabla_p T_{aq}\varphi_{pbq},\\
(K_3)_{ab} &= \nabla_p T_{qa} \varphi_{pqb},\\
\langle\nabla T,\psi\rangle_l &= \nabla_i T_{jk} \psi_{ijkl},\\
(T\oct T)_{pq} &= T_{im} T_{jn} \varphi_{ijp}\varphi_{mnq}.
\end{align*}
\end{proposition}
Using the contraction identities we obtain the useful identities
\begin{align}
T^2 - T\circ T^t - T\circ \Pop (T)&= T\circ (\Vop T\lrcorner \varphi),\label{eqn:id1}\\
|T|^2-\langle T,T^t\rangle  -\langle T,\Pop(T)\rangle&=|\Vop T|^2 = -\tr (T\circ (\Vop \lrcorner \varphi)). \label{eqn:id2}
\end{align}
To prove \eqref{eqn:id1}, we compute
\begin{align*}
T_{jm} T_{mk} - T_{jm} T_{km} - (T \circ \Pop(T))_{jk} &= T_{jm} T_{ab} (g_{am} g_{bk} - g_{ak} g_{bm}-\psi_{abmk}),\\
&=T_{jm} T_{ab} \varphi_{abl} \varphi_{mkl},\\
&=T_{jm}(\Vop T)_l \varphi_{lmk},
\end{align*}
and taking the trace proves \eqref{eqn:id2}.

We can thus rewrite the expression of the scalar curvature in terms of the torsion in Proposition \ref{prop:g2bianchi_identities} as
\begin{equation} \label{eqn:scalar_torsion}
\scal = -2\Div \Vop T +(\tr T)^2 + |\Vop T|^2 -|T|^2.
\end{equation}

\subsection{Variations of $G_2$-structures}
Let $\{\varphi_t\}_{t\in (-\epsilon,\epsilon)}$ be a smooth $1$-parameter family of $G_2$-structures on $M$, inducing the family of Riemannian metrics $\{g_t\}_{t\in (-\epsilon,\epsilon)}$, such that $\varphi_0=\varphi$, $g=g_0$ and $\psi=*_g \varphi$. Then by \eqref{eqn:hX_decomposition} we have that
\begin{align} \label{variationofaG2structure}
\left. \frac{\partial}{\partial t}\right|_{t=0} \varphi_t = h \diamond \varphi + X \lrcorner \psi
\end{align}
for some symmetric $2$-tensor $h$ and $1$-form $X$ on $M$, where the diamond operator $\diamond$ is defined in terms of $\varphi$.

The induced infinitesimal variation of the Riemannian metrics $g_t$, suppressing the parameter $t$, is then computed in \cite{flows1} as
\begin{equation}\label{eqn:metric_evol}
\left. \frac{\partial}{\partial t}\right|_{t=0} g_{ij} = 2h_{ij},
\end{equation}
while the induced variation of the inverse of the metric is given by
$$ \left. \frac{\partial}{\partial t}\right|_{t=0}  g^{ij} = -2h^{ij}, $$
where $h^{ij}=g^{ia}g^{jb}h_{ab}$. 

The variation of the volume form $ d \mu $ is
\begin{align*}
\left. \frac{\partial}{\partial t}\right|_{t=0}  d \mu = \tr h d \mu. 
\end{align*}
Also from \cite{flows1}, the variation the dual forms $\psi_t$ is given by
\begin{align*}
\left. \frac{\partial}{\partial t}\right|_{t=0} \psi_{ijkl}= & h_{im}\psi_{mjkl} + h_{jm} \psi_{imkl} + h_{km}\psi_{ijml} + h_{lm} \psi_{ijkm}\\
& -X_i \phi_{jkl} + X_j \phi_{ikl} - X_k \phi_{ijl} + X_l \phi_{ijk}.
\end{align*} 
and the variation of the torsion tensor of a $G_2$-structure is given by
\begin{equation}\label{eqn:torsion_variation}
\left. \frac{\partial}{\partial t}\right|_{t=0}  T_{ij} =  T_{ia}h_{aj} + T_{ia}X_l \varphi_{laj} + \nabla_ah_{bi}\phi_{abj} + \nabla_i X_j.
\end{equation}

A special, and very important, instance of a variation of a $G_2$-structure is the one induced by the flow of a vector $V$ on $M$. A direct computation (see for instance Section 2.7 in \cite{flows2}) leads to the identity
\begin{equation}\label{eqn:Lie_phi}
\mathcal L_{V}\varphi = \frac{1}{2}\mathcal L_{V} g \diamond \varphi + \left(-\frac{1}{2}\curl(V) +V\lrcorner T\right) \lrcorner \psi.
\end{equation}
In particular, the case $V=\Vop T$ will play a crucial role throughout this paper. From Corollary 5.28 in \cite{flows2} and Proposition \ref{prop:g2bianchi_identities} we know that
\begin{equation*}
\begin{aligned}
\curl( \Vop T)&= \Div T^t - \Div T+T^t(\Vop T) - T(\Vop T),\\
&=\nabla \tr T -\Div T+T^t (\Vop T).
\end{aligned}
\end{equation*}
Therefore, we obtain the identity
\begin{equation}\label{eqn:LieVopT}
\mathcal L_{\Vop T} \varphi = \frac{1}{2}\mathcal L_{\Vop T} g \diamond \varphi +\left(\frac{1}{2}\left(\Div T- \nabla \tr T +T^t(\Vop T)\right) \right) \lrcorner\psi.
\end{equation}

\subsection{Quasilinear second order differential operators on $G_2$-structures.}\label{sec:diff_ops}

Given a $7$-manifold $M$ admitting $G_2$-structures, by \cite{flows2} there are only six quasilinear differential operators of second-order acting on the space $\Omega^3_{+}$ of the $G_2$-structures on $M$, that correspond to an element of the tangent space $\Omega^3=\Omega^3_{1+27}\oplus \Omega^3_7$ of $\Omega^3_{+}$, whose linearizations have linearly independent principal symbols. 

These are the symmetric $2$-tensors $\ric$, $\scal\; g$, $\mathcal L_{\Vop T} g$, $F$, where $F_{jk}=R_{abcd}\varphi_{abj}\varphi_{cdk}$, and the $1$-forms $\Div T$ and $\nabla \tr T$. By Proposition \ref{prop:g2bianchi_identities} they can all be expressed in terms of the covariant derivative of the torsion $T$ of a $G_2$-structure. Moreover, note that by Proposition \ref{prop:g2bianchi_identities}, $\nabla \tr T$, $\Div T^t$ coincide up to lower order terms, so they do not have linearly independent principal symbols.

\begin{definition}
A (non-linear) differential operator $P:\Omega^3_+\rightarrow \Omega^3$ is called special Ricci-like operator if
\begin{equation*}
P(\varphi)=(-\ric+a\mathcal L_{\Vop T} g+\textrm{lots})\diamond \varphi+((1+a) \Div T-a\nabla \tr T +\textrm{lots})\lrcorner \psi,
\end{equation*}
where $\textrm{``lots"}$ above stands for ``lower order terms", namely terms that depend only on the torsion of $\varphi$, which we require to be quadratic on the torsion, by scaling considerations.
\end{definition}
The main motivation to focus on these operators is that their principal symbol has a behaviour very similar to the symbol the Ricci tensor as an operator on Riemannian metrics. 

To describe this behaviour, let $x\in M$, $\xi \in T^*_x M$, $\xi\not = 0$ and define the operator 
$$ B_\xi: S^2(T^*_x M)\oplus T^*_x M \rightarrow S^2(T^*_x M)\oplus T^*_x M$$ by
$$B_\xi(h,X) = (1+a) h(\xi) - \left(a+\frac{1}{2}\right) \xi \tr h - a \Vop(\xi\otimes X).$$ 

\begin{proposition}[Proposition 6.42 in \cite{flows2}]\label{prop:special_rl_symbol}
Let $M$ be a $7$-manifold with a $G_2$-structure $\varphi$. Then, for any $x\in M$, $\xi\in T^*_x M$ and $(h,X)\in S^2(T^*_x M)\oplus T^*_x M$,
\begin{equation*}
\begin{aligned}
&\sigma_\xi(D_\varphi P) (h\diamond \varphi+X\lrcorner\psi)= \left(|\xi|^2 h + ( \xi \otimes B_\xi(h,X))_\sym\right)\diamond\varphi +\left(|\xi|^2 X +\Vop(\xi \otimes B_\xi(h,X))\right)\lrcorner \psi.
\end{aligned}
\end{equation*}

In particular $\sigma_\xi (D_\varphi P)$, viewed as an operator in $S^2(T_x^*M)\oplus T_x^*M$ via the isomorphism \eqref{eqn:isomorphism}, leaves $\ker  B_\xi$ invariant, its kernel is transverse to $\ker B_\xi$ and is only due to the diffeomorphism invariance of $P$. Moreover, $$\sigma_\xi (D_\varphi P)(h\diamond \varphi +X\lrcorner \psi) = |\xi|^2 (h\diamond \varphi +X\lrcorner \psi),$$
for every $(h,X)\in\ker B_\xi$.
\end{proposition}

As we mentioned in the introduction, our interest in using special-Ricci-like operators to define a flow of $G_2$-structures comes from the expectation that the behaviour described in Proposition \ref{prop:special_rl_symbol} implies that geometric quantities evolve, modulo lower order terms, by the heat equation. Below we verify that this is indeed the case, at least for the scalar curvature and the torsion of the evolving $G_2$-structure.

\begin{proposition}
Suppose that the $1$-parameter family $\varphi(t)$, $t\in I$, of $G_2$-structures, induces the family of Riemannian metrics $g(t)$ and evolves by 
\begin{equation}\label{eqn:srl_flow}
\frac{\partial \varphi}{\partial t} = (-\ric +a\mathcal L_{\Vop T} g +lots)\diamond \varphi + ((1+a) \Div T -a \nabla \tr T + lots)\lrcorner \psi.
\end{equation}
Then the scalar curvature and the torsion of the evolving $G_2$-structures satisfy evolution equations of the form
\begin{align}\label{eqn:scal_srl}
\left(\frac{\partial}{\partial t} -\Delta_{g(t)} \right) \scal &= lots,\\
\left(\frac{\partial}{\partial t} -\Delta_{g(t)} \right) T_{ij}&= lots.\label{eqn:torsion_srl}
\end{align}
\end{proposition}

\begin{proof}
Under \eqref{eqn:srl_flow}, $\frac{\partial g}{\partial t} = 2h$ for $h=-\ric +a \mathcal L_{\Vop T} g +lots$, thus we obtain (neglecting terms of lower order)
\begin{equation*}
\frac{\partial}{\partial t} \scal = 2(-\Delta \tr h + \Div\Div h -\langle h,\ric\rangle)= \Delta \scal +lots.
\end{equation*}
By \eqref{eqn:scalar_torsion} and the twice contracted second Bianchi identity,
\begin{align*}
\tr h &= -\scal +2a \Div \Vop T,\\
&=-(1+a)\scal +lots,\\
\Div h &= -\Div \ric + a \Div \mathcal L_{\Vop T} g,\\
&= -\frac{1}{2}\nabla \scal + a(\Delta \Vop T + \nabla _i \nabla_j \Vop T_j),
\end{align*}
hence
\begin{equation*}
\Div \Div h = -\frac{1}{2}\Delta \scal + 2a \Delta \Div \Vop T +lots = \left(-\frac{1}{2} -a\right)\Delta \scal + lots.
\end{equation*}

To compute the evolution of the torsion, we know from \eqref{eqn:torsion_variation}, again neglecting lower order terms, that
\begin{equation}\label{eqn:torsion_evol_hots}
\frac{\partial T_{ij}}{\partial t} = \nabla_a h_{bi} \varphi_{abj} + \nabla_i X_j + lots,
\end{equation}
for $X= (1+a)\Div T-a \nabla \tr T +lots$. 

Then,
\begin{equation}\label{eqn:torsion_srl_flow}
\begin{aligned}
\nabla_a h_{bi} \varphi_{abj} +\nabla_i X_j&= - \nabla_a R_{bi} \varphi_{abj} + a\nabla_a\mathcal L_{\Vop T} g_{bi} \varphi_{abj} +\nabla_i X_j+lots,\\
&=-\nabla_a R_{bi} \varphi_{abj} + \nabla_i ( a(\nabla \tr T - \Div T))_j +\nabla_i X_j + lots,\\
&=-\nabla_a R_{bi} \varphi_{abj} +\nabla_i \Div T_j,
\end{aligned}
\end{equation}
since, commuting derivatives, using the first contraction identity in \eqref{eqn:g2identities_1}, and Proposition \ref{prop:g2bianchi_identities} we obtain that
\begin{align*}
\nabla_a \mathcal L_{\Vop T} g_{bi} \varphi_{abj} &= \nabla_a \nabla_b T_{pq} \varphi_{pqi} \varphi_{abj}+\nabla_a\nabla_i T_{pq}\varphi_{pqb}\varphi_{jab} + lots,\\
&=\frac{1}{2}(R_{abmp} T_{mq} +R_{abmq} T_{pm}) \varphi_{pqi} \varphi_{abj}+\nabla_a \nabla_i T_{pq} (g_{pj} g_{qa} - g_{pa} g_{qj} - \psi_{pqja}) + lots,\\
&=\nabla_a \nabla_i T_{ja} - \nabla_a \nabla_i T_{aj} + lots,\\\
&=\nabla_i (\Div T^t_j  - \Div T_j) + lots,\\
&=\nabla_i (\nabla_j \tr T - \Div T_j) +lots.
\end{align*}
On the other hand, using the $G_2$-Bianchi identity \eqref{eqn:g2_bianchi}, commuting derivatives and then using  the second Bianchi identity, we obtain
\begin{equation}\label{eqn:torsion_laplacian}
\begin{aligned}
\Delta T_{ij} &= \nabla_a \nabla _a T_{ij} \\
&=\nabla_a \left( \nabla_i T_{aj} +T_{ap}T_{iq} \varphi_{pqj} +\frac{1}{2} R_{aipq}\varphi_{pqj} \right) + lots,\\
&=\nabla_i \nabla_a T_{aj} + \frac{1}{2}\nabla_a R_{aipq} \varphi_{pqj} + lots,\\
&=\nabla_i \Div T_j +\frac{1}{2}( -\nabla_p R_{aiqa} -\nabla_q R_{aiap}) \varphi_{pqj} + lots,\\
&=\nabla_i \Div T_j -\nabla_p R_{qi} \varphi_{pqj} + lots.
\end{aligned}
\end{equation}
Equation \eqref{eqn:torsion_srl} then follows from \eqref{eqn:torsion_evol_hots}, \eqref{eqn:torsion_srl_flow} and \eqref{eqn:torsion_laplacian}.
\end{proof}

\section{Functionals in $G_2$-geometry}

\subsection{Basic functionals in $G_2$-geometry}

\begin{proposition}
\label{variationofbasicfunctionals}
Let $(\varphi_t)_{t\in (-\epsilon,\epsilon)}$ be a smooth $1$-parameter family of $G_2$-structures on $M$ such that 
$$\left.\frac{d}{dt}\right|_{t=0}\varphi_t= h\diamond\varphi + X\lrcorner \psi.$$
Then
\begin{align*}
\left.\frac{d}{dt}\right|_{t=0} \int_M \scal d \mu &=  \int_M \langle h, \scal g-2\ric \rangle d \mu,\\
\left.\frac{d}{dt}\right|_{t=0} \int_M  (\tr T)^2 d\mu &= \int_M \langle h, (\tr T)^2 g - 2\tr T T_{\sym}\rangle d\mu + \int_M \langle X, -2\nabla \tr T - 2\tr T VT\rangle d\mu,\\
\left.\frac{d}{dt}\right|_{t=0} \int_M  |T|^2 d\mu &=\int_M \langle h, 2\ric +\mathcal L_{\Vop T} g +|T|^2 g - 2\tr T T_\sym + 2 (T\circ (\Vop T\lrcorner \varphi) )_{\sym}\rangle d\mu\\
&+\int_M \langle X, -2\Div T\rangle d\mu,\\
 \left.\frac{d}{dt}\right|_{t=0} \int_M  \langle T,T^t\rangle d\mu&=\int_M \langle h, \frac{1}{2} F +\langle T,T^t\rangle g + (T\oct T)_\sym - 2(T^2)_\sym+2 (\Pop (T) T)_\sym\rangle d\mu\\
    &+\int_M \langle X,-2\Div T^t -2 \Vop(T^2)\rangle d\mu,\\
   \left.\frac{d}{dt}\right|_{t=0} \int_M  \langle T,\Pop(T)\rangle d\mu&= \int_M \langle h, 2\ric - \frac{1}{2}F - \scal g+(\tr T)^2 g - \langle T,T^t\rangle g\\
    &-(T\oct T)_\sym - 2\tr T T_\sym+2(T^2)_\sym - 2(\Pop (T)T)_\sym\rangle d\mu\\
    &+\int_M \langle X, -2(\tr T) \Vop T + 2\Vop (T^2) + 2T(\Vop T)\rangle d\mu,\\
   \left.\frac{d}{dt}\right|_{t=0}\int_M |\Vop T|^2 d\mu &=\int_M \langle h, \scal\ g+\mathcal{L}_{\Vop T} g - (\tr T)^2g + |T|^2 g+2 (T\circ (\Vop T\lrcorner \varphi))_\sym\rangle d\mu \\
    &+ \int_M \langle X, -2\Div T +2\nabla\tr T + 2\tr T \Vop T  \rangle d\mu.
\end{align*}
Here $F_{pq}=R_{ijkl}\varphi_{ijp}\varphi_{klq}$ and $(A\oct A)_{pq}=A_{im} A_{jn} \varphi_{ijp}\varphi_{mnq}$. 
\end{proposition}
\begin{proof}
 The first formula follows from the standard first variation formula of the total scalar curvature functional, see \cite{besse}. The remaining variation formulae follow trivially from Corollary 5.35 in \cite{flows2} and \eqref{eqn:id2}. However, here we have used  Proposition \ref{prop:g2bianchi_identities} and \eqref{eqn:id1} in order to present the result in a more compact form.
\end{proof}

\begin{remark}
Given a $G_2$-structure on a $7$-manifold $M$, there are two natural ways to equip the tangent space $T_\varphi \Omega^3_+=\Omega^3$ with an inner-product, depending on whether we view a $3$-form as a skew-symmetric tensor, or as isomorphic to  a section of $S^2(T^*M)\oplus T^*M$, see Remark \ref{rmk:L2_ips}. In particular, given $\alpha,\beta \in \Omega^3$ with
\begin{align*}
\alpha &=h\diamond\varphi +X\lrcorner \psi,\\
\beta &= w\diamond \varphi+Y\lrcorner \psi,
\end{align*}
we can define
\begin{equation}\label{eqn:L2_inner_products}
\begin{aligned}
(\alpha,\beta)_{L^2}& = \int_M \langle \alpha,\beta \rangle d\mu,\\
\langle \alpha,\beta\rangle_{L^2} &= \int_M \langle h,w\rangle +\langle X,Y\rangle d\mu.
\end{aligned}
\end{equation}
The particular choice among \eqref{eqn:L2_inner_products} affects the calculation of the gradient of each of the functionals of Proposition \ref{variationofbasicfunctionals}. For instance, we have that
\begin{equation*}
\langle (\scal g -2\ric) \diamond \varphi, h\diamond \varphi \rangle_{L^2} = \int_M \langle \scal g -2\ric, h\rangle d\mu,
\end{equation*}
so its $\langle\cdot,\cdot\rangle_{L^2}$ gradient is the $3$-form $(\scal g -2\ric)\diamond \varphi$. On the other hand, denoting by $\ric_0, h_0$ the trace-free parts of $\ric$ and $h$ respectively, we compute
\begin{align*}
((\scal g -2\ric) \diamond \varphi, h\diamond \varphi )_{L^2}&=\int_M 270 \scal \; \tr h +12 \langle \ric_0, h_0\rangle d\mu,\\
&=\int_M \langle 270\scal g , \frac{\tr h}{7} g\rangle +\langle 12 \ric_0,h_0\rangle d\mu,\\
&=\int_M \langle 270 \scal g +12 \ric_0, h\rangle d\mu,
\end{align*}
by Remark \ref{rmk:L2_ips}. Therefore, $(\scal g - 2\ric) \diamond \varphi$ is not the $(\cdot,\cdot)_{L^2}$ gradient of the total scalar curvature functional. Because of this, it is more natural to use $\langle\cdot,\cdot\rangle_{L^2}$ in what follows.
\end{remark}

\subsection{Diffeomorphism invariance} \label{subsection:diffeo_invariance}
All the functionals studied in the previous section are diffeomorphism invariant, namely the are smooth maps $\mathcal F: \Omega^3_+ \rightarrow \mathbb R$ with the property that
\begin{equation*}
\mathcal F(\Phi^*\varphi) = \mathcal F(\varphi),
\end{equation*}
for any diffeomorphism $\Phi$ of $M$.

Suppose now that we have such a diffeomorphism invariant functional on the space of $G_2$-structures, and that there are operators 
\begin{align*}
&Q_1:\Omega^3_+\rightarrow\mathcal S^2\\
&Q_2: \Omega^3_+ \rightarrow \Omega^1
\end{align*}
such that if $\omega= h\diamond\varphi+X\lrcorner \psi$, according to the decomposition \eqref{eqn:hX_decomposition}, then the first variation of $\mathcal F$ at $\varphi$ is given by
\begin{equation}
\begin{aligned}
D_\varphi \mathcal F(\omega) &= \langle Q_1(\varphi) \diamond \varphi + Q_2(\varphi)\lrcorner \psi, h\diamond \varphi +X\lrcorner \psi\rangle_{L^2},\\
&=\int_M \langle Q_1,h\rangle +\langle Q_2, X\rangle d\mu,
\end{aligned}
\end{equation}
where we write $Q_i=Q_i(\varphi)$, to simplify the notation.

By the diffeomorphism invariance of $\mathcal F$, however, we know that $D\mathcal F_\varphi(\mathcal L_Y \varphi)=0$ for any vector field $Y$ on $M$. Therefore, by \eqref{eqn:Lie_phi}, we obtain the identity
\begin{equation*}
\int_M \langle Q_1,\frac{1}{2}\mathcal L_Y g\rangle +\langle Q_2,-\frac{1}{2} \curl Y +Y\lrcorner T\rangle d\mu =0,
\end{equation*}
which, by integration by parts, gives \begin{equation*}
\int_M \langle \Div Q_1 +\frac{1}{2}\curl Q_2-\frac{1}{2} Q_2 \lrcorner \Pop(T) - T(Q_2), Y\rangle d\mu=0,
\end{equation*}
for every $Y$. 

It follows that $(Q_1,Q_2)\in \mathcal S^2\times \Omega^1$ is in the kernel of the operator 
\begin{align*}
&\Lop:\mathcal S^2\times \Omega^1\rightarrow \Omega^1\\
\Lop(h,X)&= \Div h +\frac{1}{2}\curl X -\frac{1}{2} X\lrcorner \Pop (T) - T(X).
\end{align*}
Moreover, the $L^2$-adjoint $\Lop^*:\Omega^1\rightarrow \mathcal S^2\times \Omega^1$ of $\Lop$ is given by
\begin{equation*}
\Lop^*(Y) = \left( -\frac{1}{2}\mathcal L_Y g, \frac{1}{2}\curl Y - Y\lrcorner T\right),
\end{equation*}
namely, the expression of $-\mathcal L_Y \varphi$ with respect to the decomposition \eqref{eqn:hX_decomposition}.

The orthogonal subspaces $\mathrm{Im} \Lop^*$ and $\ker \Lop$ correspond to directions tangent and orthogonal, respectively, to the action on the space of $G_2$-structures of the diffeomorphism group of $M$. In fact, the following decomposition holds.

\begin{proposition} \label{Firstsplitting}
Given a $G_2$-structure $\varphi$ on a $7$-manifold $M$, the tangent space $T_\varphi \Omega^3_+$ of the space of $G_2$-structures $\Omega^3_+$, equipped with the inner product $\langle\cdot,\cdot\rangle_{L^2}$, is isometric to $\mathrm{Im} \Lop^* \oplus \ker \Lop$, where the direct sum is $\langle\cdot,\cdot\rangle_{L^2}$-orthogonal.
 \end{proposition}
The proof of the proposition follows in the same line of reasoning as the infinitesimal slice theorem in Riemannian geometry, see  Chapter 4 in \cite{besse}.
\begin{proof}

The operator $\Lop\circ \Lop^*$ is clearly self-adjoint, and in fact elliptic. To see this, neglecting lower order terms in $V$, we compute:
\begin{align*}
\Lop\circ \Lop^* (V)_k &=-\frac{1}{2} \nabla_i  (\nabla_i V_k +\nabla_k V_i) +\frac{1}{4} \nabla_p (\curl V)_q \varphi_{pqk} + lots\\
&=-\frac{1}{2} \Delta V_k -\frac{1}{2} \nabla_k \nabla_i V_i +\frac{1}{4} \nabla_p \nabla_a V_b \varphi_{abq}\varphi_{kpq} + lots\\
&=-\frac{1}{2} \Delta V_k -\frac{1}{2} \nabla_k \nabla_i V_i +\frac{1}{4} \nabla_p \nabla_a V_b (g_{ak} g_{bp} - g_{ap} g_{bk}-\psi_{abkp}) + lots\\
&=-\frac{3}{4} \Delta V_k -\frac{1}{4} \nabla_k \nabla_i V_i   + lots.
\end{align*}
Therefore, the principal symbol of $\Lop\circ \Lop^*$ is
\begin{equation*}
\sigma_\xi (\Lop\circ \Lop^*) (V) = -\frac{3}{4} |\xi|^2 V- \frac{1}{4} \xi_k \langle V,\xi \rangle.
\end{equation*}
Let $V\in \ker \sigma_\xi (\Lop\circ \Lop^*)$. We then have that
\begin{equation*}
-\frac{3}{4} |\xi|^2 \langle V,\xi\rangle - \frac{1}{4}|\xi|^2  \langle V,\xi\rangle = -|\xi|^2 \langle V,\xi\rangle =0.
\end{equation*}
This implies that $\langle V,\xi\rangle =0$, since $\xi\not =0$, hence 
$$\sigma_\xi (\Lop\circ \Lop^*) (V) = -\frac{3}{4} |\xi|^2 V =0$$
which implies that $V=0$.

Therefore, $\Lop\circ \Lop^*$ is an elliptic operator, as $\sigma_\xi (\Lop\circ \Lop^*)$ is an isomorphism.

Now, observe that $\ker \Lop \circ \Lop^* = \ker \Lop^* $. To see this, first note that the inclusion $\ker \Lop^* \subset \ker \Lop\circ \Lop^* $ is trivial. For the reverse inclusion, let $V \in \ker \Lop \circ \Lop^* $. Then 
    \begin{align*}
        0 = \int_M \langle \Lop \circ \Lop^* (V), V \rangle d \mu   = \int_M |\Lop^*(V)|^2 d \mu 
    \end{align*}
    Thus $V \in \ker \Lop^*$. 

To finish the proof of the proposition, let $(h,X) \in \mathcal S^2(T^*M) \oplus \Omega^1(M)$. It suffices to prove that the equation $L \circ L^* (V) =L(h,X)$ has a unique solution $V \in \Omega^1(M)$, since then 
   $$(h,X)=\underbrace{(h,X) - \Lop^*(V)}_{\ker \Lop} +\underbrace{\Lop^*(V)}_{\mathrm{Im} \Lop^*},$$
 which proves that $\mathcal S^2(T^*M)\times \Omega^1(M) \subset \ker \Lop \oplus \mathrm{Im} \Lop^*$.
   
 By Fredholm theory, this equation has a solution if and only if $\Lop(h,X)$ is orthogonal to $\ker \Lop \circ \Lop^* = \ker  \Lop^*$. This is indeed the case, since if $V \in \ker \Lop^*$, we have that
    \begin{equation*}
        \int_M \langle \Lop(h,X), V \rangle d \mu =  \int_M \langle (h,X) , \Lop^*(V) \rangle d \mu =0.
    \end{equation*}
\end{proof}

\subsection{Special Ricci-like operators and diffeomorphism invariance} In this subsection we will use the machinery developed in Subsection \ref{subsection:diffeo_invariance} to investigate which special Ricci-like operators are allowed in the highest order terms of a diffeomorphism invariant functional on $G_2$-structures.

\begin{proposition}\label{prop:uniqueness}
Let $M$ be a $7$-manifold with a $G_2$-structure $\varphi$ and let $\psi=*\varphi$. Suppose that $\mathcal F:\Omega^3_{+} \rightarrow \mathbb R$ is a diffeomorphism invariant functional and  that there are second order quasilinear differential operators $Q_1: \Omega^3_{+}\rightarrow \mathcal S^2$ and $Q_2: \Omega^3_{+} \rightarrow \Omega^1$ such that for any variation $\{\varphi_t\}_{t\in (-\varepsilon,\varepsilon)}$ of $G_2$-structures with $\varphi_0=\varphi$ and $\left.\frac{d\varphi_t}{dt}\right|_{t=0}= h\diamond \varphi + X\lrcorner \psi$
\begin{equation*}
\left.\frac{d}{dt}\right|_{t=0} \mathcal F(\varphi_t) =\int_M \langle h,Q_1(\varphi)\rangle +\langle X,Q_2(\varphi)\rangle d\mu_{g_{\varphi}}.
\end{equation*}
\begin{enumerate}
\item $Q_1$ and $Q_2$ are of the form
\begin{align*}
    Q_1&=\alpha \ric +\beta \scal g +\gamma \mathcal L_{\Vop T} g+ \zeta F+lots\\
    Q_2&= \delta \Div T + \epsilon \nabla \tr T +lots
\end{align*}
and
\begin{equation}\label{eqn:hots_fcn}
\frac{\alpha}{2} +\beta -\frac{\gamma}{2}+\frac{\delta}{4} =0 \quad \textrm{and}\quad \gamma +\frac{\delta}{2}=0.
\end{equation}
This system describes, modulo scaling, a $3$-parameter family of second order quasilinear differential operators, as the only possible second order terms of $Q_1$ and $Q_2$.

\item If $Q_1$ and $Q_2$ have the special form
\begin{align*}
    Q_1&=-\ric +a \mathcal L_{\Vop T} g+ \beta \scal g +lots\\
    Q_2&= (1+a) \Div T - a \nabla \tr T +lots
\end{align*}
then $a=-\frac{1}{3}$ and $\beta=\frac{1}{6}$. In particular, it is impossible for a special Ricci-like operator (i.e. $\beta=0$ above) to be the gradient, up to lower order terms, of the functional $\mathcal F$.
 \end{enumerate}
\end{proposition}
\begin{proof}
By the diffeomorphism invariance of the functional and the classification in \cite{flows2} of second order quasilinear operators, see Subsection \ref{sec:diff_ops}, we know that
\begin{align*}
Q_1&= \alpha \ric +\beta \scal g +\gamma \mathcal L_{\Vop T} g +\zeta F +lots\\
Q_2&= \delta \Div T+\epsilon \nabla \tr T + lots.
\end{align*} 
Moreover, by Subsection \ref{subsection:diffeo_invariance} we know that $Q_1$ and $Q_2$ should satisfy $\Lop(Q_1,Q_2)=0$, namely
\begin{equation*}
\Div Q_1 +\frac{1}{2} \curl Q_2 - \frac{1}{2} Q_2 \lrcorner \Pop(T) -T(Q_2)=0
\end{equation*}
In particular, the term $\Div Q_2+\frac{1}{2}\curl Q_2$ which a-priori is of order $3$ with respect to $\varphi$ has to reduce to an expression of lower order. To exploit this, it suffices to compute $\Div Q_1$ and $\curl Q_2$ omitting terms of order less than $3$. 

We already know that $\Div (\scal g)= \nabla \scal$ and $\Div \ric = \frac{1}{2} \nabla \scal$, by the twice contracted second Bianchi identity. Expanding the terms in $\Div \mathcal L_{\Vop T} g$ we obtain
\begin{align*}
\nabla_i (\mathcal{L}_{\Vop T}g_{ij} ) = & \nabla_i (\nabla_i (T_{ab} \varphi_{abj}) + \nabla_j (T_{ab}\varphi_{abi})  )\\
=& \nabla_i (\nabla_i T_{ab} \varphi_{abj} +T_{ab}T_{im}\psi_{mabj} +\nabla_j T_{ab}\varphi_{abi}+T_{ab}T_{jm}\psi_{mabi})\\
=&\Delta T_{ab} \varphi_{abj}+ \nabla_i\nabla_j T_{ab} \varphi_{abi} +T_{im} T_{ab}\psi_{abmj}+T_{jm}T_{ab}\psi_{abmi}\\
=&\Delta T_{ab} \varphi_{abj}+ \nabla_i\nabla_j T_{ab} \varphi_{abi} +T_{im}\Pop(T)_{mj}+T_{jm}\Pop(T)_{mi}.
\end{align*}
Commuting derivatives and applying the $G_2$-Bianchi identity \eqref{eqn:g2_bianchi} we obtain, omitting lower order terms,
\begin{align*}
\nabla_i (\mathcal{L}_{\Vop T}g_{ij} )= & \Delta T_{ab} \varphi_{abj}+ \nabla_j\nabla_i T_{ab}\varphi_{abi}+lots\\
=&\Delta T_{ab} \varphi_{abj} +\frac{1}{2}\nabla_j \left(( T_{ip}T_{aq}+\frac{1}{2} R_{iapq})\varphi_{pqb}\right) \varphi_{abi}+lots\\
=&\Delta T_{ab} \varphi_{abj} +\frac{1}{2}\nabla_j \left( T_{ip}T_{aq}+\frac{1}{2} R_{iapq}\right)\varphi_{pqb} \varphi_{iab}+lots\\
=&\Delta T_{ab}\varphi_{abj} +\frac{1}{4} \nabla_j R_{iapq} (g_{pi} g_{qa} - g_{pa} g_{qi} -\psi_{pqia})+lots\\
=&\Delta T_{ab} \varphi_{abj} -\frac{1}{4}\nabla_j \scal - \frac{1}{4}\nabla_j\scal -\nabla_j R_{iapq}\psi_{pqia}+lots\\
=&\Delta T_{ab}\varphi_{abj} -\frac{1}{2}\nabla_j \scal +lots.
\end{align*}
Moreover, 
\begin{equation*}
    \nabla_i F_{ik} =\nabla_i (R_{abpq}\varphi_{abi}\varphi_{pqk})= \nabla_i R_{abpq} \varphi_{abi}\varphi_{pqk}+lots=lots,
\end{equation*}
since $\nabla_i R_{abpq} \varphi_{abi}=0$ by the second Bianchi identity.

Therefore,
\begin{equation*}
\Div Q_1 = \Div (\alpha \ric +\beta\scal g+\gamma\mathcal L_{\Vop T} g +\zeta F) = \left(\frac{\alpha}{2} +\beta -\frac{\gamma}{2}\right) \nabla \scal + \gamma \Delta T_{ab} \phi_{abj}+lots.
\end{equation*}
Similarly we compute, omitting the lower order terms,
\begin{align*}
\curl(\Div T)_k&=\nabla_i \nabla_b T_{bj} \varphi_{ijk}=\nabla_b \nabla_i T_{bj} \varphi_{ijk} +lots\\
&=\nabla_b \nabla_b T_{ij} \varphi_{ijk} +\frac{1}{2}\nabla_b R_{ibpq}\varphi_{pqj} \varphi_{kij} +lots\\
&=\Delta T_{ij}\varphi_{ijk} +\frac{1}{2} \nabla_b R_{ibpq}(g_{pk}g_{qi}-g_{pi}g_{qk} -\psi_{pqki}) +lots\\
&=\Delta _{ij}\varphi_{ijk} +\nabla_b R_{bk} +lots\\
&=\Delta T_{ij} \varphi_{ijk} +\frac{1}{2}\nabla _k\scal+lots.
\end{align*}
Therefore, since $\curl(\nabla \tr T)=0$, we obtain $\curl(Q_2)_k = \delta \Delta T_{ij}\phi_{ijk}$, so finally

\begin{equation*}
\Div(Q_1)_k+\frac{1}{2}\curl( Q_2 )_k=\left(\frac{\alpha}{2} +\beta -\frac{\gamma}{2}+\frac{\delta}{4}\right) \nabla_k \scal +\left( \gamma +\frac{\delta}{2}\right)\Delta T_{ij} \phi_{ijk} +lots.
\end{equation*}
We conclude that $\alpha,\beta,\gamma,\delta,\epsilon, \zeta$ should satisfy
\begin{align*}
\frac{\alpha}{2} +\beta -\frac{\gamma}{2}+\frac{\delta}{4} &=0,\\
\gamma +\frac{\delta}{2}&=0.
\end{align*}
The system has a $4$-dimensional set of solutions, from which the first claim of the proposition follows. Putting $\alpha=-1, \gamma=a, \delta= 1+a, \epsilon=-a$ and $\zeta=0$ gives that $a=-\frac{1}{3}$ and $\beta=\frac{1}{6}$, which proves the second claim of the proposition.
\end{proof}

\begin{remark}\label{rmk:functionals_EL}
It is easy to verify that all the functionals in Proposition \ref{variationofbasicfunctionals} satisfy the system \eqref{eqn:hots_fcn}. Moreover, any solution of the system is of the form 
\begin{equation}\label{eqn:spanning}
(\alpha,\beta,\gamma,\delta,\epsilon,\zeta)=\beta(-2,1,0,0,0)+\delta (-1,0,-\frac{1}{2},0,0) +\epsilon(0,0,0,1,0) +\zeta (0,0,0,0,1).
\end{equation}

Observe that by Proposition \ref{variationofbasicfunctionals}, each of the vectors in the right-hand side of \eqref{eqn:spanning} corresponds, in the same order, to the functionals 
\begin{equation*}
\int_M \scal d\mu,\quad -\frac{1}{2}\int_M |T|^2 d\mu, \quad -\frac{1}{2}\int_M (\tr T)^2 d\mu, \quad \frac{1}{2}\int_M (\langle T,T^t\rangle - (\tr T)^2 )d\mu.
\end{equation*}
\end{remark}

\section{The $G_2$-Hilbert functional} \label{sec:g2_hilbert}

In this section we seek to define a diffeomorphism invariant functional
 $$\mathcal F(\varphi) =\int_M \Fop(\varphi) d\mu_{g_\varphi}$$
  on the space of $G_2$-structures, which will be motivated by the Einstein-Hilbert functional in Riemannian Geometry. Namely, in analogy with the Einstein-Hilbert functional, we require that: 
  \begin{enumerate}
  \item $\Fop (\varphi)$ is a second order scalar quasilinear operator on $G_2$-structures, which is linear in scalar curvature (or equivalently, by Proposition \ref{prop:g2bianchi_identities}, linear in $\Div \Vop T$) and quadratic on the torsion of $\varphi$.
  \item The first variation of $\Fop$ is of the form
$$D_\varphi \Fop (h\diamond\varphi + X\lrcorner \psi) = \langle Q_1,h\rangle +\langle Q_2,X\rangle + \textrm{divergence terms},$$
where $Q_1,Q_2$ are second order quasilinear differential operators such that the linearization of $P=Q_1 \diamond\varphi + Q_2\lrcorner \psi$, at least restricted to the right gauge, has principal symbol
\begin{equation*}
\sigma_\xi (D_\varphi P)(h\diamond\varphi +X\lrcorner\psi) = |\xi|^2 (h\diamond\varphi +X\lrcorner \psi).
\end{equation*}
\end{enumerate}
By Proposition \ref{prop:special_rl_symbol}, good candidates for such operators are the special Ricci-like operators. Thus, we will require that, for some $a\in \mathbb R$, 
\begin{align*}
Q_1 &= -\ric +a \mathcal L_{\Vop T} g + lots,\\
Q_2 & =(1+a) \Div T - a\nabla \tr T +lots.
\end{align*}
Therefore, we are seeking a diffeomorphism invariant functional on $G_2$-structures for which the first variation will be of the form
\begin{equation}\label{eqn:first_var_general}
\begin{aligned}
D\mathcal F_\varphi (h\diamond\varphi + X\lrcorner\psi) &= \int_M \langle Q_1 + \Fop(\varphi) g, h \rangle +\langle Q_2, X\rangle d\mu\\
&=\int_M \langle h,-\ric +a\mathcal L_{\Vop T} g + \beta \scal g +lots\rangle d\mu \\
&+\int_M \langle X, (1+a)\Div T-a \nabla \tr T +lots\rangle d\mu.
\end{aligned}
\end{equation}

By Proposition \ref{prop:uniqueness} however, if such functional exists, we know that $a=-\frac{1}{3}$ and $\beta=\frac{1}{6}$.

We can now easily come up with a functional with the required properties using Remark \ref{rmk:functionals_EL}. Namely, the operators in \eqref{eqn:first_var_general} correspond to the choices $\alpha=-1$, $\beta=\frac{1}{6}$, $\gamma=-\frac{1}{3}$, $\delta=\frac{2}{3}$, $\epsilon=\frac{1}{3}$, $\zeta=0$. 

The corresponding functional then is
\begin{equation}\label{eqn:g2_hilbert}
\mathcal F(\varphi)=\int_M \frac{1}{6} \scal  -\frac{1}{3} |T|^2  -\frac{1}{6}  (\tr T)^2 d\mu. 
\end{equation}

We will refer to $\mathcal F$ as the $G_2$-Hilbert functional.

Using \eqref{eqn:scalar_torsion} then it is easy to see that
\begin{equation}\label{eqn:F_second_form}
\mathcal F(\varphi)= \int_M\left( -\frac{1}{2} |T|^2 +\frac{1}{6} |\Vop T|^2 \right)d\mu,
\end{equation}
since
\begin{equation}\label{eqn:two_forms}
\frac{1}{6} \scal -\frac{1}{3} |T|^2 -\frac{1}{6} (\tr T)^2 = -\frac{1}{3} \Div \Vop T -\frac{1}{2} |T|^2 +\frac{1}{6} |\Vop T|^2.
\end{equation}

\begin{proposition}\label{prop:variation_lagrangian}
Let $M$ be a $7$-manifold with a $G_2$-structure $\varphi$ inducing the Riemannian metric $g$ on $M$, and let $\psi=*\varphi$. Let $\{\varphi_t\}_{t\in (-\varepsilon,\varepsilon)}$ be a family of $G_2$-structures on $M$ such that $\varphi_0=\varphi$ and
$$\left.\frac{d}{dt}\right|_{t=0} \varphi_t = h\diamond \varphi + X\lrcorner \psi$$
for some symmetric $2$-tensor $h$ and $1$-form $X$. Then
\begin{align*}
&\left.\frac{d}{dt}\right|_{t=0} \left(\frac{1}{6} \scal -\frac{1}{6}(\tr T)^2 -\frac{1}{3} |T|^2\right)=\langle h,\hat P_1\rangle +\langle X,P_2\rangle + \Div A,
\end{align*}
where 
\begin{align*}
\hat P_1&= -\ric -\frac{1}{3}\mathcal L_{\Vop T} g -\frac{2}{3} (T\circ (\Vop T\lrcorner \varphi))_\sym + \tr T T_\sym,\\
P_2 &=  \frac{2}{3}  \Div T + \frac{1}{3}  \nabla \tr T + \frac{1}{3}\tr T \Vop T,\\
A&=-\frac{1}{3} \nabla \tr h + \frac{1}{3}  \Div h -\frac{2}{3} \Vop(h\circ T) -\frac{2}{3} T(X)-\frac{1}{3} \tr T X.
\end{align*}
\end{proposition}
\begin{proof}
We need to compute the variation of $|T|^2=T_{pq} T_{ab}g^{pa}g^{qb}$, $(\tr T)$ and $\scal$, collecting all terms involving derivatives of $h$ and $X$ into divergence terms.

By \eqref{eqn:torsion_variation}, we obtain
\begin{equation*}
 \begin{aligned}
 \frac{\partial}{\partial t} |T|^2&=2\frac{\partial}{\partial t} T_{pq} T_{pq} - 2h_{pa}T_{pq}T_{aq} - 2h_{qb}T_{pq}T_{pb}\\
 &=2(T_{pm}h_{mq} + T_{pm}X_l\varphi_{lmq} +\nabla_m h_{lp} \varphi_{mlq} +\nabla_p X_q)T_{pq} \\
 &-2\langle h, T\circ T^t\rangle - 2 \langle h, T^t\circ T\rangle\\
 &=2\langle T,\curl(h)\rangle +2\langle\nabla X,T\rangle - 2\langle h, T\circ T^t\rangle.
 \end{aligned}
 \end{equation*}
 Now, clearly 
 \begin{equation*}
 \langle\nabla X, T\rangle= \nabla_a X_b T_{ab}=\nabla_a (X_b T_{ab}) - \langle X, \Div T\rangle
 \end{equation*}
 and
  \begin{align*}
 \langle T,\curl(h)\rangle &= T_{pq} \nabla_a h_{bp}\varphi_{abq}\\
 &=\nabla_a (T_{pq} h_{bp} \varphi_{abq}) - \nabla_a T_{pq} \varphi_{abq} h_{bp} - T_{pq}T_{am}\psi_{mabq} h_{bp}\\
 &=\nabla_a (T_{pq} h_{bp} \varphi_{abq}) -\left(\nabla_p T_{aq} +T_{ai}T_{pj}\varphi_{ijq}+\frac{1}{2} R_{apij}\varphi_{ijq} \right)\varphi_{abq} h_{bp}\\
 &-T_{pq}\Pop(T)_{qb} h_{bp},\\
 &=\nabla_a (T_{pq} h_{bp} \varphi_{abq}) +\nabla_p (T_{aq} \varphi_{aqb}) h_{bp} - T_{aq} T_{pm}\psi_{maqb} h_{bp}\\
 &-T_{ai} T_{pj} (g_{ia} g_{jb} - g_{ib} g_{ja}-\psi_{ijab}) h_{bp}\\
 &-\frac{1}{2} R_{apij} (g_{ia} g_{jb} - g_{ib} g_{ja}-\psi_{ijab}) h_{bp}-T_{pq}\Pop(T)_{qb} h_{bp},\\
 &=\nabla_a (T_{pq} h_{bp}\varphi_{abq}) +\nabla_p \Vop T_b h_{bp}- T_{pm}\Pop(T)_{mb} h_{bp}\\
 &-\tr T T_{pb} h_{bp} + (T^2)_{pb} h_{bp} + T_{pj} \Pop (T)_{jb} h_{bp}\\
 &-\frac{1}{2}(-R_{pb} -R_{pb}) h_{bp} - T_{pq} \Pop(T)_{qb} h_{bp},\\
 &=\nabla_a (T_{pq} h_{bp} \varphi_{abq}) + \langle h, \frac{1}{2}\mathcal L_{\Vop T} g\rangle -\langle h, T\circ \Pop(T)\rangle\\
 &-\langle h, \tr T T\rangle + \langle h, T^2\rangle +\langle \ric , h\rangle,\\
 &=\langle h, \ric+\frac{1}{2}\mathcal L_{\Vop T} g- \tr T T_{\sym} - (T\circ \Pop(T))_{\sym} + (T^2)_{\sym} \rangle \\
 &+ \nabla_a (T_{pq} h_{bp} \varphi_{abq}).
 \end{align*}
 
 Therefore,
 \begin{align*}
 &\frac{\partial}{\partial t} |T|^2 =\\
 &=\langle h, 2\ric +\mathcal L_{VT} g - 2\tr T T_{\sym} - 2 (T\circ P(T))_{\sym} +2 (T^2)_{\sym} - 2 T\circ T^t\rangle\\
 &+\langle X, -2\Div T\rangle+2\nabla_a (T_{pq} h_{bp} \varphi_{abq}) +2\nabla_a(X_b T_{ab}).
 \end{align*}
 Moreover, using \eqref{eqn:id1} we obtain the more elegant form
 \begin{align*}
\frac{\partial}{\partial t} |T|^2&= \langle h, 2\ric +\mathcal L_{\Vop T} g - 2\tr T T_{\sym} +2T\circ (\Vop T\lrcorner \varphi)\rangle+\langle X, -2\Div T\rangle\\
&+2\nabla_a (T_{pq} h_{bp} \varphi_{abq}) +2\nabla_a(X_b T_{ab}).
 \end{align*}

For the variation of $(\tr T)^2$ we compute
\begin{align*}
&\frac{\partial}{\partial t} (\tr T)^2 = 2 \tr T \frac{\partial}{\partial t} \tr T\\
&=-4\tr T \langle h, T\rangle +2\tr T \tr\left(\frac{\partial}{\partial t} T\right)\\
&=-4\tr T \langle h, T\rangle +2\tr T (T_{pm} h_{mp} + T_{pm} X_l \varphi_{lmp} +\nabla_m h_{lp}\varphi_{mlp} + \Div X)\\
&=-2 \langle h, \tr T T_{\sym}\rangle -2\tr T\langle X,\Vop T\rangle + 2\tr T \Div X\\
&=-2 \langle h, \tr T T_{\sym}\rangle-2\tr T\langle X, \Vop T\rangle +2\nabla_a(\tr T X_a) - 2\langle\nabla \tr T,X\rangle,\\
&=\Div (2 \tr T X) -2\langle X, \nabla \tr T +\tr T \Vop T\rangle -2 \langle h, \tr T T_{\sym}\rangle
\end{align*}
Therefore, since $\left.\frac{\partial g}{\partial t}\right|_{t=0}=2h$, combining the computations above with \eqref{eqn:var_scal}
we obtain
\begin{align*}
&\frac{\partial}{\partial t} \left( \frac{1}{6} \scal - \frac{1}{3} |T|^2 - \frac{1}{6}(\tr T)^2\right)=\\
&=\frac{2}{6} \left(-\Delta \tr h + \Div \Div h - \langle h, \ric\rangle \right)\\
&-\frac{1}{3} \left( \langle h,2\ric +\mathcal L_{VT} g -2\tr T T_{\sym} +2T\circ (VT\lrcorner \varphi)\rangle\right.\\
&\left.+ \langle X,-2\Div T\rangle \right.\\
&\left.+ 2\nabla_a(T_{pq} h_{bp}\varphi_{abq} +X_b T_{ab}) \right)\\
&-\frac{2}{6}\left( \Div (\tr T X) -\langle X,\nabla \tr T +\tr T \Vop T\rangle - \langle h, \tr T T_{\sym}\rangle\right)\\
&=\langle h, -\ric-\frac{1}{3}\mathcal L_{\Vop T} g + \tr T T_{\sym} - \frac{2}{3} T\circ (\Vop T \lrcorner \varphi)\rangle\\
&+\langle X, \frac{2}{3} \Div T +\frac{1}{3} \nabla \tr T + \frac{1}{3} \tr T \Vop T \rangle  \\
&+\Div A,
\end{align*}
where
\begin{align*}
A_a&=- \frac{1}{3} \nabla_a \tr h + \frac{1}{3}  \Div h_a -\frac{2}{3} T_{pq}h_{bp} \varphi_{abq} -\frac{2}{3} X_b T_{ab} -\frac{1}{3} \tr T X_a\\
&=-\frac{1}{3} \nabla_a \tr h + \frac{1}{3}  \Div h_a -\frac{2}{3} \Vop(h\circ T)_a -\frac{2}{3} X_b T_{ab} -\frac{1}{3} \tr T X_a
\end{align*}

\end{proof}
Therefore the $G_2$-Hilbert functional indeed fulfills the requirements set at the beginning of this section. The first variation of the $G_2$-Hilbert functional can now be directly computed by Proposition \ref{prop:variation_lagrangian}.

\begin{proposition} \label{firstvariationofF}
Let $M$ be a $7$-manifold with a $G_2$-structure $\varphi$ inducing the Riemannian metric $g$ on $M$, and let $\psi=*\varphi$. Let $\{\varphi_t\}_{t\in (-\varepsilon,\varepsilon)}$ be a family of $G_2$-structures on $M$ such that $\varphi_0=\varphi$ and
$$\left.\frac{d}{dt}\right|_{t=0} \varphi_t = h\diamond \varphi + X\lrcorner \psi$$
for some symmetric $2$-tensor $h$ and $1$-form $X$. Then 
\begin{align*}
\left. \frac{d}{d t}\right|_{t=0}  \mathcal{F}(\varphi_t) &=   \int_M \langle h, P_1(\varphi) \rangle d \mu_g  + \int_M \langle X , P_2(\varphi) \rangle d \mu_{g_\varphi} \\
&=\int_M \langle h, \hat P_1 +\left(\frac{1}{6}\scal - \frac{1}{6}(\tr T)^2 -\frac{1}{3} |T|^2\right) g\rangle +\langle X,P_2\rangle d\mu
\end{align*}
where 
\begin{equation*}
\begin{aligned} 
P_1 &=  -\ric -\frac{1}{3}\mathcal L_{\Vop T} g -\frac{2}{3} (T\circ (\Vop T\lrcorner \varphi))_\sym + \tr T T_\sym +\left(\frac{1}{6} \scal  -\frac{1}{6} (\tr T)^2 -\frac{1}{3} |T|^2 \right) g, \\
P_2 &=  \frac{2}{3}  \Div T + \frac{1}{3}  \nabla \tr T + \frac{1}{3}\tr T \Vop T.
\end{aligned}
\end{equation*}
Therefore, the $\langle \cdot,\cdot\rangle_{L^2}$-gradient of the $G_2$-Hilbert functional is given by the operator
$$P(\varphi)=P_1(\varphi) \diamond \varphi+P_2(\varphi)\lrcorner\psi.$$
\end{proposition}

The volume-normalized $G_2$-Hilbert functional $\mathcal{\tilde F}$ is defined as
\begin{equation}\label{volumenormalizedF}
 \mathcal{\tilde{F}}(\varphi) = \frac{1}{\vol(g_\varphi)^{\frac{5}{7}}} \mathcal F(\varphi).  
\end{equation}

\begin{proposition} \label{prop:norm_F_variation}
Let $M$ be a $7$-manifold with a $G_2$-structure $\varphi$ inducing the Riemannian metric $g$ on $M$, and let $\psi=*\varphi$. Let $\{\varphi_t\}_{t\in (-\varepsilon,\varepsilon)}$ be a family of $G_2$-structures on $M$ such that $\varphi_0=\varphi$ and
$$\left.\frac{d}{dt}\right|_{t=0} \varphi_t = h\diamond \varphi + X\lrcorner \psi$$
for some symmetric $2$-tensor $h$ and $1$-form $X$. Then

\begin{align*}
\left. \frac{\partial}{\partial t}\right|_{t=0} \tilde{\mathcal F}(\varphi_t) = \vol(g)^{-\frac{5}{7}} \left( \int_M \langle h, P_1 - \frac{5}{7} \vol(g)^{-1}\mathcal F(\varphi)  g \rangle d \mu_g + \int_M \langle X, P_2 \rangle d \mu_g \right).
\end{align*}
\end{proposition}

\begin{definition}
We will call a $G_2$-structure $\varphi$ on a compact $7$-manifold $M$ a \textit{static $G_2$-structure} if it is a critical point of the volume-normalized $G_2$-Hilbert functional, namely if 
\begin{align*}
P_1(\varphi) &= \lambda g,\\
P_2(\varphi)&=0,
\end{align*}
for some real number $\lambda$.
\end{definition}

We finish this subsection, by analyzing some properties of static $G_2$-structures.

\begin{proposition} \label{criticalpointstheorem}
Torsion free $G_2$ structures are critical points of the functional $\mathcal{F}$ while nearly $G_2$ structures are critical points of the normalized functional $\tilde{\mathcal{F}}$.
\end{proposition}

\begin{proof}
When $\varphi$ is torsion free or nearly-$G_2$, there is a constant $c$ such that its torsion satisfies $T=c g$. In particular, $\Vop T=0$, and $P_2(\varphi)=0$ since $\nabla T=0$.
 Moreover, using Proposition \ref{prop:g2bianchi_identities} we obtain
\begin{align*}
\ric=6c^2 g,\quad
\scal= 42 c^2,\quad
\tr T T_\sym= 7c^2 g,\quad
(\tr T)^2= 49 c^2,\quad
|T|^2 =7 c^2. 
\end{align*}
It is then immediate that
\begin{equation*}
    P_1 =  -\ric -\frac{1}{3}\mathcal L_{\Vop T} g -\frac{2}{3} (T\circ (\Vop T\lrcorner \phi))_\sym + \tr T T_\sym +\frac{1}{6} \scal g -\frac{1}{6} (\tr T)^2 g-\frac{1}{3} |T|^2 g= -\frac{5}{2} c^2 g.
\end{equation*}
and 
\begin{equation*}
\frac{5}{7}\vol(g)^{-1} \mathcal F(\varphi) = \frac{5}{7} \vol(g)^{-1} \int_M  -\frac{1}{2}|T|^2 d\mu_g =\frac{5}{7}\vol(g)^{-1}\int_M -\frac{7}{2}c^2 d\mu_g=-\frac{5c^2}{2}.
\end{equation*}
Therefore,  $P_1-\frac{5}{7} \vol(g)^{-1} \mathcal F(\varphi) g=0$ and $P_2=0$, so indeed by Proposition \ref{prop:norm_F_variation}, $\varphi$ is a critical point for the normalized $G_2$-Hilbert functional.

\end{proof}

\begin{proposition}
 Let $\varphi$ be a $G_2$-structure  on a $7$-manifold $M$.
 \begin{enumerate}
 \item If $\Vop T=0$ then $\mathcal F(\varphi)\leq 0$, and $\mathcal F(\varphi)=0$ if and only if $\varphi$ is torsion free.
 \item If $\scal=0$ then $\mathcal F(\varphi)\leq 0$, and $\mathcal F(\varphi)=0$ if and only if $\varphi$ is torsion free.
 \item If $\varphi$ is a critical point of $\mathcal F$ and either $\Vop T=0$ or $\scal = 0$, then it is torsion free. In particular, critical points of $\mathcal F$ which are either closed, co-closed or scalar flat $G_2$-structures are torsion free.
 \end{enumerate}
\end{proposition}
\begin{proof}
Assertions (1) and (2) follow immediately from
$$\mathcal F=\int_M -\frac{1}{2}|T|^2 +\frac{1}{6} |\Vop T|^2 d\mu = \int_M \frac{1}{6}\scal -\frac{1}{3}|T|^2 -\frac{1}{6} (\tr T)^2 d\mu.$$
To obtain Assertion (3) we need to also observe that at a critical point  $\mathcal F=0$.

The statement about critical points that are closed $G_2$-structures follow from the fact that if $\varphi$ is closed then $T\in \Omega^2_{14}$ hence $\Vop T=0$. Similarly, $\Vop T=0$ for co-closed $G_2$-structures. See for instance \cite{flows1}.

\end{proof}

\subsection{Uniqueness of the $G_2$-Hilbert functional} In this subsection we will prove that the quantity
\begin{equation*}
\Fop=\frac{1}{6}\scal - \frac{1}{6} (\tr T)^2 -\frac{1}{3}|T|^2,
\end{equation*}
and the operators
\begin{align*}
\hat P_1 &= -\ric -\frac{1}{3}\mathcal L_{\Vop T} g -\frac{2}{3} (T\circ (\Vop T\lrcorner \varphi))_\sym +\tr T T_\sym,\\
P_2 &= \frac{2}{3} \Div T+\frac{1}{3} \nabla \tr T +\frac{1}{3} \tr T \Vop T,
\end{align*}
are uniquely determined by the following assumptions.
\begin{enumerate}
\item $\Fop$ is linear in the scalar curvature and quadratic in the torsion of $\varphi$.
\item The variation of $\Fop$ is given by 
\begin{equation*}
\left. \frac{\partial}{\partial t}\right|_{t=0} \Fop = \langle \hat P_1, h\rangle + \langle P_2,X\rangle + \Div A,
\end{equation*}
where $A$ is a $1$-from, and  $\varphi\mapsto \hat P_1\diamond \varphi+P_2\lrcorner \psi$ is a special Ricci-like operator, namely $\hat P_1$, $P_2$ are of the form
\begin{align*}
\hat P_1&= -\ric +a \mathcal L_{\Vop T} g + lot,\\
P_2&= (1+a) \Div T+a \nabla \tr T +lot,
\end{align*}
for some $a\in \mathbb R$.
\end{enumerate}
It is easy to see that Assumption 1 implies that $\Fop$ must be a linear combination of the terms
\begin{equation*}
\scal, \quad (\tr T)^2, \quad |T|^2, \quad \langle T, T^t\rangle, \quad \langle T,\Pop(T)\rangle, \quad |\Vop T|^2,
\end{equation*}
which define the functionals in Proposition \ref{variationofbasicfunctionals}. In fact, using Proposition \ref{variationofbasicfunctionals}, or by direct computation, we see that if $\varphi_t$, $t\in (-\epsilon,\epsilon)$ is a smooth variation of $G_2$-structures with $\left.\frac{\partial \phi_t}{\partial t}\right|_{t=0} = h\diamond\varphi+X\lrcorner \psi$, then
\begin{align*}
\left.\frac{\partial}{\partial t}\right|_{t=0} \scal &= \langle h,-2\ric\rangle + \textrm{divergence terms},\\
\left.\frac{\partial}{\partial t}\right|_{t=0} (\tr T)^2 &= \langle h, -2\tr T T_\sym\rangle + \langle X, -2\nabla \tr T -2\tr T \Vop T\rangle +\textrm{divergence terms},\\
\left.\frac{\partial}{\partial t}\right|_{t=0} |T|^2 &= \langle h, 2\ric +\mathcal L_{\Vop T} g -2\tr T T_\sym +2(T\circ(\Vop T\lrcorner \varphi))_\sym \rangle \\
&+ \langle X, -2\Div T\rangle +\textrm{divegence terms},\\
\left.\frac{\partial}{\partial t}\right|_{t=0} \left(\langle T,T^t\rangle + \langle T,\Pop(T)\rangle \right)&=\langle h, 2\ric - \scal\; g +(\tr T)^2 g -2\tr T T_\sym -(\langle T,T^t\rangle + \langle T,\Pop(T)\rangle)g \rangle\\
&+\langle X, -2\nabla \tr T -2\tr T \Vop T\rangle + \textrm{divergence terms},\\
(\textrm{by identity } \eqref{eqn:id2}) \quad&=\langle h, 2\ric - \scal\; g +(\tr T)^2 g -2\tr T T_\sym -|T|^2 g +|\Vop T|^2 g\rangle\\
&+\langle X, -2\nabla \tr T -2\tr T \Vop T\rangle + \textrm{divergence terms},\\
\left.\frac{\partial}{\partial t}\right|_{t=0} |\Vop T|^2 &=\langle h, (\scal - (\tr T)^2 +|T|^2 - |\Vop T|^2)g +\mathcal L_{\Vop T} g +2(T\circ (\Vop T\lrcorner \varphi))_\sym\rangle\\
&+\langle X, -2\Div T +2\nabla \tr T +2\tr T \Vop T\rangle + \textrm{divegence terms}.
\end{align*}
Notice that we don't need to include the terms $\langle T,T^t\rangle$ and $\langle T,\Pop(T)\rangle$ separetly, since their variation will include the symmetric operator $F_{ij}= R_{abcd}\varphi_{abi}\varphi_{cdj}$, which doesn't appear in special Ricci-like operators.

It follows that the variation of any linear combination
$$\Fop= \alpha\scal + \beta (\tr T)^2 + \gamma |T|^2 + \delta( \langle T,T^t\rangle + \langle T,\Pop(T)\rangle ) +\varepsilon |\Vop T|^2$$
is given by
\begin{align*}
&\left.\frac{\partial}{\partial t}\right|_{t=0} \Fop =\\
&=\langle h, (-2\alpha +2\gamma +2\varepsilon) \ric + (\gamma +\delta) \mathcal L_{\Vop T} g + (\delta-\varepsilon) \scal \; g \\
&\quad+ (-2\beta -2\gamma-2\varepsilon) \tr T T_\sym + (2\gamma+2\delta)(T\circ(\Vop T\lrcorner \varphi))_\sym +(\delta-\varepsilon)\left(- (\tr  T)^2 +|T|^2 -|\Vop T|^2\right)g \rangle\\
&+\langle X, (-2\gamma -2\delta) \Div T +(-2\beta +2\delta -2\varepsilon) \nabla\tr T +(-2\beta+2\delta -2\varepsilon) \tr T \Vop T \rangle,\\
&+\textrm{divergence terms}.
\end{align*}
Since we require the highest order terms of the operators in the above expression to be special Ricci-like, we need
\begin{equation}\label{eqn:uniqueness_system}
\begin{aligned}
-2\alpha +2\gamma +2\varepsilon &= -1,\\
\gamma+\delta &= a,\\
\delta &=\varepsilon,\\
-2\gamma -2\delta &= 1+a,\\
-2\beta +2\delta -2\varepsilon &=- a.
\end{aligned}
\end{equation}
The system \eqref{eqn:uniqueness_system} has a $1$-dimensional set of solutions where
\begin{equation*}
a=-\frac{1}{3}, \quad \alpha=\frac{1}{6}, \quad \beta=-\frac{1}{6},\quad \gamma= -\delta -\frac{1}{3},\quad  \varepsilon=\delta,
\end{equation*}
which correspond to the quantities 
$$\Fop_\delta=\frac{1}{6}\scal -\frac{1}{6}(\tr T)^2 -\left(\delta+\frac{1}{3}\right) |T|^2 +\delta (\langle T,T^t\rangle +\langle T,\Pop(T)\rangle)+\delta |\Vop T|^2.$$
However, using identity \eqref{eqn:id2}, we  compute
\begin{align*}
\Fop_\delta &= \frac{1}{6}\scal -\frac{1}{6}(\tr T)^2 -\frac{1}{3} |T|^2+\delta \underbrace{(-|T|^2 + \langle T,T^t\rangle +\langle T,\Pop(T)\rangle+|\Vop T|^2)}_{=0, \textrm{ by } \eqref{eqn:id2}},\\
&=\frac{1}{6}\scal -\frac{1}{6}(\tr T)^2 - \frac{1}{3} |T|^2.
\end{align*}

It follows that Assumptions 1 and 2 determine the quantity
$$\Fop=\frac{1}{6}\scal -\frac{1}{6}(\tr T)^2 - \frac{1}{3} |T|^2$$
uniquely. Thus, the variation of $\Fop$  computed in Proposition \ref{prop:variation_lagrangian} also uniquely distinguishes the operators $\hat P_1$ and $P_2$.

\subsection{Bianchi-type identity} The twice contracted second Bianchi identity 
\begin{equation}\label{eqn:2nd_bianchi}
\Div\left( \ric - \frac{\scal}{2} g\right)=0
\end{equation}
is a consequence of the diffeomorphism invariance of the Einstein--Hilbert functional, since $\mathcal G=-\ric+\frac{\scal}{2}g$  is its gradient, and thus $\Div \mathcal G=0$. Equation \eqref{eqn:2nd_bianchi} can be restated in terms of the Bianchi operator $\beta_g:\mathcal S^2\rightarrow \Omega^1$,
\begin{equation*}
\beta_g(h)=\Div\left(h-\frac{\tr_g h}{2} g\right),
\end{equation*}
since it implies that $\ric \in \ker \beta_g$. This motivates the following proposition.

\begin{proposition}\label{prop:tildeP1_decomp}
The gradient of the $G_2$-Hilbert functional can be expressed in the form 
$$P=\left(\tilde P_1 -\frac{\tr \tilde P_1}{4} g \right)\diamond\varphi + P_2\lrcorner \psi,$$
where $\tilde P_1$ is the special Ricci-like operator given by
$$\tilde P_1=\hat P_1 +\frac{1}{3}\left(|T|^2 - \frac{1}{3} |\Vop T|^2\right) g.$$
Moreover, $\tilde P_1$ is the unique special Ricci-like operator and $\lambda=-\frac{1}{4}$ the unique real number such that $P_1= \tilde P_1 +\lambda \tr \tilde P_1 g$.
\end{proposition}
\begin{proof}
Suppose that for some $\lambda\in\mathbb R$, $P_1$ decomposes as $P_1= \tilde P_1+\lambda\tr \tilde P_1 g$. Taking traces, we obtain that $\tr P_1 = (7\lambda+1) \tr \tilde P_1$, therefore
\begin{equation}\label{eqn:tildeP1_derivation}
\tilde P_1 = P_1 -\lambda \tr \tilde P_1 g = P_1 -\frac{\lambda}{7\lambda+1} \tr P_1 g,
\end{equation}
so $\tilde P_1$ is uniquely determined by $P_1$ and $\lambda$.

Recall that
$$P_1 =  -\ric -\frac{1}{3}\mathcal L_{\Vop T} g -\frac{2}{3} (T\circ (\Vop T\lrcorner \varphi))_\sym + \tr T T_\sym +\left(\frac{1}{6} \scal  -\frac{1}{6} (\tr T)^2 -\frac{1}{3} |T|^2 \right) g,$$
hence, by \eqref{eqn:scalar_torsion},
\begin{equation}\label{eqn:trP1}
\begin{aligned}
\tr P_1 &= -\scal - \frac{2}{3} \Div \Vop T +\frac{2}{3} |\Vop T|^2 + (\tr T)^2 +\frac{7}{6} \scal -\frac{7}{6} (\tr T)^2 -\frac{7}{3} |T|^2,\\
&=\frac{1}{2} \scal-\frac{1}{2}(\tr T)^2+\frac{1}{3} |\Vop T|^2 -2|T|^2.
\end{aligned}
\end{equation}
It follows that the second order terms of $\tilde P_1$ are given by
$$-\ric -\frac{1}{3}\mathcal L_{\Vop T} g +\left(\frac{1}{6}  -\frac{\lambda}{2(7\lambda+1)}\right)\scal.$$
Since $\tilde P_1$ is a special Ricci-like operator if and only if the coefficient of $\scal$ vanishes, it follows that $\lambda=-\frac{1}{4}$ is the only choice for which this happens.

A direct computation then using \eqref{eqn:tildeP1_derivation} shows that
\begin{equation}\label{eqn:tildeP1}
\tilde P_1 = -\ric -\frac{1}{3}\mathcal L_{\Vop T} g -\frac{2}{3} (T\circ (\Vop T\lrcorner \varphi))_\sym + \tr T T_\sym + \frac{1}{3} \left(|T|^2 -\frac{1}{3} |\Vop T|^2\right)g.
\end{equation}
\end{proof}
By Subsection \ref{subsection:diffeo_invariance}, we know that $\Lop(P_1,P_2)=0$, which by Proposition \ref{prop:tildeP1_decomp} immediately gives the following corollary.
\begin{corollary}\label{cor:tildeBianchi}
The operator $\tilde P(\varphi)=\tilde P_1(\varphi) \diamond \varphi +P_2(\varphi)\lrcorner \psi$ satisfies the Bianchi-type identity
$\tilde B(\tilde P_1,P_2)=0$, where
$$\tilde B(h,X) = \Lop \left(h-\frac{\tr h}{4} g,X\right)=\Div h -\frac{1}{4}\nabla \tr h +\frac{1}{2}\curl(X) -\frac{1}{2} X\lrcorner \Pop(T) - T(X).$$
\end{corollary}

\begin{remark}
By Proposition \ref{prop:special_rl_symbol}, the principal symbol of the operator $D_\varphi P$
is given by 
\begin{equation*}
\begin{aligned}
&\sigma_\xi (D_\varphi P)(h\diamond \varphi+X\lrcorner\psi)=\\
&=(|\xi|^2 h + (\xi \otimes  B_\xi(h,X))_\sym)\diamond\varphi + (|\xi|^2 X +\Vop (\xi\otimes  B_\xi(h,X)),
\end{aligned}
\end{equation*}
where $B_\xi: \mathcal S^2 \times \Omega^1 \rightarrow \Omega^1$ is the bi-linear map defined by
\begin{align*}
B_\xi(h,X)_k&=\frac{2}{3}h(\xi) -\frac{1}{6} \xi \tr h +\frac{1}{3} \Vop( \xi \otimes  X)=\frac{2}{3}\left( h(\xi)  - \frac{1}{4} \xi \tr h+\frac{1}{2} \Vop(\xi \otimes X)\right).
\end{align*}
Observe that $\tilde B_\xi (h,X) =h(\xi)  - \frac{1}{4} \xi \tr h+\frac{1}{2} \Vop(\xi \otimes X)$
is the principal symbol of the linear differential operator $\tilde B$ defined in Corollary \ref{cor:tildeBianchi}. This strongly resembles the relation of the Bianchi operator $\beta_g$ with the linearlization $-\frac{1}{2} D\ric$, since
$$-\frac{1}{2} D_g\ric (h) = \Delta_g h +\mathcal L_{\beta_g(h)} g + lots$$
\end{remark}

\subsection{Alternative expressions of the $G_2$-Hilbert functional}

By further applying \eqref{eqn:scalar_torsion} into \eqref{eqn:trP1} we obtain that
\begin{equation}\label{eqn:trP1_equiv}
\begin{aligned}
\tr P_1 &= \frac{1}{2}(\underbrace{\scal - (\tr T)^2}_{\eqref{eqn:scalar_torsion}}) +\frac{1}{3} |\Vop T|^2 -2|T|^2\\
&=-\Div \Vop T+\frac{1}{2}\left(|\Vop T|^2 -|T|^2\right) + \frac{1}{3}|\Vop T|^2 - 2|T|^2\\
&=-\Div \Vop T+\frac{5}{6} |\Vop T|^2 -\frac{5}{2} |T|^2\\
&=-\Div \Vop T+5\left(-\frac{1}{2}|T|^2 +\frac{1}{6} |\Vop T|^2\right).
\end{aligned}
\end{equation}
Moreover,
\begin{equation}\label{eqn:tr_hat_P1}
\begin{aligned}
\tr \hat P_1 &= \underbrace{(\tr T)^2-\scal}_{\eqref{eqn:scalar_torsion}} -\frac{2}{3}\Div \Vop T +\frac{2}{3} |\Vop T|^2,\\
&=\frac{4}{3} \Div \Vop T -\frac{1}{3} \underbrace{|\Vop T|^2}_{\eqref{eqn:scalar_torsion}}+|T|^2\\
&=\frac{4}{3} \Div \Vop T -\frac{1}{3} \left(\scal -(\tr T)^2 +2\Div \Vop T +|T|^2\right) + |T|^2\\
&=\frac{2}{3} \Div \Vop T -2 \left(\frac{1}{6}\scal -\frac{1}{6}(\tr T)^2 - \frac{1}{3} |T|^2\right). 
\end{aligned}
\end{equation}
and similarly, by \eqref{eqn:tildeP1},
\begin{equation*}
\tr \tilde P_1 = \frac{4}{3}\Div \Vop T -\frac{20}{3} \left(-\frac{1}{2}|T|^2 +\frac{1}{6} |\Vop T|^2\right).
\end{equation*}
Therefore, by \eqref{eqn:g2_hilbert} and \eqref{eqn:F_second_form}, we obtain 
\begin{equation*}
\mathcal F=-\frac{1}{2}\int_M \tr \hat P_1 d\mu =-\frac{3}{20} \int_{M} \tr \tilde P_1 d\mu=\frac{1}{5}\int_M  \tr P_1 d\mu.
\end{equation*}

\subsection{Flows of $G_2$-structures} \label{sec:flows} In the previous subsections we saw that the $G_2$-Hilbert functional, and the associated operators 
\begin{align*}
\hat P(\varphi)&= \left(-\ric -\frac{1}{3}\mathcal L_{\Vop T} g -\frac{2}{3} (T\circ (\Vop T \lrcorner \varphi))_\sym +\tr T T_\sym\right)\diamond \varphi \\
&+ \left(\frac{2}{3} \Div T +\frac{1}{3} \nabla \tr T +\frac{1}{3} \tr T \Vop T\right)\lrcorner \psi,\\
\tilde P(\varphi)&= \left(-\ric -\frac{1}{3}\mathcal L_{\Vop T} g -\frac{2}{3} (T\circ (\Vop T \lrcorner \varphi))_\sym +\tr T T_\sym +\frac{1}{3} \left(|T|^2 -\frac{1}{3} |\Vop T|^2\right)g\right)\diamond \varphi \\
&+ \left(\frac{2}{3} \Div T +\frac{1}{3} \nabla \tr T +\frac{1}{3} \tr T \Vop T\right)\lrcorner \psi,
\end{align*}
arise naturally from a collection of principles which in the context of Riemannian Geometry uniquely identify the Ricci tensor. Thus, from this point of view, natural analogues of the Ricci flow 
$$\frac{\partial g}{\partial t} =-2\ric_g$$
in $G_2$-geometry would be the flows
\begin{align}
\frac{\partial \varphi}{\partial t} = \hat P(\varphi),\label{eqn:hat_flow}\\
\frac{\partial \varphi}{\partial t} = \tilde P(\varphi).\label{eqn:tilde_flow}
\end{align}
Since by \eqref{eqn:LieVopT},
\begin{equation}
\mathcal L_{\Vop T} \varphi = \frac{1}{2}\mathcal L_{\Vop T} g \diamond \varphi +\left(\frac{1}{2}\left(\Div T- \nabla \tr T +T^t(\Vop T)\right) \right) \lrcorner\psi,
\end{equation}
the flows \eqref{eqn:hat_flow} and \eqref{eqn:tilde_flow} are equivalent, modulo diffeomorphisms, to the the flows induced by the operators $\hat P+\frac{2}{3}\mathcal L_{\Vop T} \varphi$ and $\tilde P+\frac{2}{3}\mathcal L_{\Vop T} \varphi$, namely
\begin{equation}
\begin{aligned}
\frac{\partial \varphi}{\partial t} &= \left(-\ric -\frac{2}{3} (T\circ (\Vop T \lrcorner \varphi))_\sym + \tr T T_\sym \right) \diamond \varphi +\left( \Div T +\frac{1}{3} \tr T \Vop T +\frac{1}{3} T^t (\Vop T)\right)\lrcorner \psi,\label{eqn:hat_flow_2}
\end{aligned}
\end{equation}
and
\begin{equation}
\begin{aligned}
\frac{\partial \varphi}{\partial t} &= \left(-\ric -\frac{2}{3} (T\circ (\Vop T \lrcorner \varphi))_\sym + \tr T T_\sym +\frac{1}{3} \left(|T|^2 -\frac{1}{3} |\Vop T|^2\right)g\right) \diamond \varphi\\ \label{eqn:tilde_flow_2}
& +\left( \Div T +\frac{1}{3} \tr T \Vop T +\frac{1}{3} T^t (\Vop T)\right)\lrcorner \psi.
\end{aligned}
\end{equation}
respectively. Note that flows \eqref{eqn:hat_flow_2} and \eqref{eqn:tilde_flow_2}, neglecting the lower terms, are a coupling of the Ricci flow with the isometric flow \eqref{eqn:isometric_flow}.

Clearly, torsion-free $G_2$-structures are fixed points of all the flows described above. On the other hand, given a nearly-$G_2$-structure $\varphi_0$ satisfying $T_{\varphi_0}=c_0 g_0$, direct computations using that $\varphi = \lambda^3 \varphi_0$ and $g=\lambda^2 g_0$ satisfy
\begin{align*}
\ric_g&= 6c_0^2 g_0,\\
T_\varphi &= c_0 \lambda g_0,\\
\Vop_\varphi T_\varphi&=0,\\
\tr_g T_\varphi&= 7\lambda^{-1} c_0,\\
|T_\varphi|^2 g &= 7c_0^2 g_0,
\end{align*}
shows that for $t\geq 0$
\begin{align*}
\varphi(t)&= (c_0^2 t+1)^3 \varphi_0,\\
\varphi(t)&=\left(\frac{10}{3} c_0^2 t+1\right)^3 \varphi_0,
\end{align*}
are solutions to \eqref{eqn:hat_flow} and \eqref{eqn:tilde_flow} respectively, as well as of \eqref{eqn:hat_flow_2} and \eqref{eqn:tilde_flow_2} since here $\Vop_\varphi T_
\varphi=0$, starting at $\varphi(0)=\varphi_0$. 

Thus, nearly-$G_2$-structure are also fixed points, modulo scaling, of the flows described above. However, unlike the Ricci flow, they are \textit{expanding} solutions of these flows, which exist for all time.

\section{The conformal Killing operator on $G_2$-structures}

\begin{definition}
The $G_2$-conformal  Killing operator $\Kop :\Omega^1 \rightarrow \mathcal S_0^2 \times \Omega^1 $, where $\mathcal S_0^2$ denotes the traceless symmetric $2$-tensors,  is defined as
\begin{equation} \label{conformalkilling}
\begin{aligned}
\Kop (Y) &= \left( -\frac{1}{2} \mathcal L_Y g + \frac{1}{7}\Div Y g , \frac{1}{2} \curl Y - Y \lrcorner T\right)\\
&=\Lop^*(Y) + \left(\frac{1}{7} \Div Y g,0\right).
\end{aligned}
\end{equation}

\end{definition}
Observe that if $V\in \ker \Kop$ then, by \eqref{eqn:Lie_phi},
$$\mathcal L_V \varphi = \left(\frac{1}{7} \Div V g\right)\diamond \varphi,$$
so the flow of $V$ generates conformal transformations of the $G_2$-structure.

The $G_2$-conformal Killing operator allows us to refine the decomposition of the tangent space of $G_2$-structures of Proposition \ref{Firstsplitting} in the following proposition.
 
\begin{proposition} \label{decompositionofG_2structures}
Given a $G_2$-structure $\varphi$ on a $7$-manifold $M$, the tangent space $T_\varphi \Omega^3_+$ of the space of $G_2$-structures $\Omega^3_+$, equipped with $\langle \cdot,\cdot\rangle_{L^2}$, is isometric to 
$$\{(fg,0), f\in C^\infty(M)\} \oplus \mathrm{Im}\Kop \oplus (\ker \Lop \cap \ker \tr),$$
where the direct sum is $\langle \cdot,\cdot\rangle_{L^2}$-orthogonal.

The $3$-forms corresponding to $\ker L\cap \ker \tr$ will be called $G_2$-transverse traceless infinitesimal deformations of the $G_2$-structure $\varphi$.
\end{proposition}

\begin{proof}
We compute, omitting lower order terms,
\begin{align*}
\Lop\circ \Kop(Y)_k &=\Lop\left(-\frac{1}{2}\mathcal L_Y g+\frac{1}{7} \Div Y g, \frac{1}{2} \curl Y -Y\lrcorner T\right)_k\\
&=\Div \left(-\frac{1}{2}\mathcal L_Y g+\frac{1}{7} \Div Y g \right) + \frac{1}{2} \curl\left(\frac{1}{2} \curl Y -Y\lrcorner T \right)_k\\
&=-\frac{1}{2} \nabla_i (\nabla_i Y_k +\nabla_k Y_i) +\frac{1}{7} \nabla_k \Div Y + \frac{1}{4} \curl \curl Y_k + lots\\
&=-\frac{1}{2} \Delta Y_k -\frac{5}{14} \nabla_k \Div Y +\frac{1}{4} \nabla_a \nabla_p Y_q \varphi_{pqb} \varphi_{kab} + lots\\
&=-\frac{1}{2} \Delta Y_k -\frac{5}{14} \nabla_k \Div Y + \frac{1}{4} \nabla_a \nabla_p Y_q (g_{pk} g_{qa} - g_{pa} g_{qk} -\psi_{pqka}) + lots \\
&=-\frac{1}{2} \Delta Y_k -\frac{5}{14} \nabla_k \Div Y +\frac{1}{4}\nabla_k \Div Y- \frac{1}{4} \Delta Y_k - \frac{1}{4}\nabla_a\nabla_p Y_q \psi_{pqka} +lots\\
&=-\frac{3}{4} \Delta Y_k -\frac{3}{28} \nabla_k \Div Y -\frac{1}{8} R_{apmq}Y_m \psi_{pqka} +lots\\
&=-\frac{3}{4} \Delta Y_k -\frac{3}{28} \nabla_k \Div Y +lots
\end{align*}
Therefore, the principal symbol of $\Lop\circ\Kop$ is given by
\begin{equation*}
\sigma_\xi (\Lop\circ\Kop) (Y)_k= -\frac{3}{4} |\xi|^2 Y_k -\frac{3}{28} \xi_k \langle \xi,Y\rangle.
\end{equation*}
As in the proof of Proposition \ref{Firstsplitting}, we conclude that $\Lop\circ \Kop$ is an elliptic operator.

Next, we observe that $\Lop\circ \Kop$ is self-adjoint.  Given $X,Y\in\Omega^1$, we compute, using \eqref{conformalkilling} and that the $2$-tensor component of $\Kop$ is always trace-free, that
\begin{equation}\label{eqn:LK_selfadjoint}
\begin{aligned}
\int_M \langle \Lop\circ \Kop(X),Y\rangle d\mu &= \int_M \langle \Kop(X), \Lop^*(Y) \rangle d\mu\\
&=\int_M \langle \Kop(X), \Lop^*(Y) +\left( \frac{1}{7} \Div Y,0\right)\rangle d\mu\\
&=\int_M \langle \Kop(X), \Kop(Y) \rangle d\mu,\\
&=\int_M \langle \Lop^*(Y) ,\Kop(Y)\rangle d\mu\\
&=\int_M \langle X, \Lop\circ \Kop(Y) \rangle d\mu.
\end{aligned}
\end{equation}
Moreover, setting $X=Y\in \ker \Lop\circ \Kop$ in \eqref{eqn:LK_selfadjoint}, we also obtain that $\ker \Lop\circ \Kop=\ker \Kop$.

To finish the proof of the proposition, take any $ (h,X) \in \mathcal S^2\times \Omega^1$. We first uniquely decompose $h=h_0+h_1$, where $\tr h_0 =0 $ and $h_1= \frac{\tr h}{7} g$. 

Therefore, $(h,X) = (h_0,X) + (h_1,0)$. It suffices to prove that the equation $\Lop\circ \Kop(V) = \Lop(h_0,X)$ has a unique solution $V\in \Omega^1(M)$, since then
$$(h_0,X) = \underbrace{(h_0,X) - \Kop(V)}_{\ker \Lop \cap \ker \tr} + \underbrace{\Kop(V)}_{\mathrm{Im}\Kop}$$
which proves that $\mathcal S^2_0 \times \Omega^1\subset \mathrm{Im} \Kop\oplus (\ker \Lop\cap \ker \tr)$.

By Fredholm theory, this equation has a unique solution if and only if $\Lop(h_0,X)$ is $L^2$-orthogonal to $ \ker \Lop\circ \Kop= \ker \Kop$. This is indeed the case, since if $V\in \ker \Kop$ then
\begin{align*}
\int_M \langle \Lop(h_0,X),V\rangle d\mu &= \int_M \langle (h_0,X), \Lop^*(V)\rangle d\mu\\
&=\int_M \langle (h_0,X),\Kop(V)\rangle d\mu =0
\end{align*}
Therefore, the required decomposition is
\begin{align*}
(h,X) = (h_1,0) +  \left((h_0 ,X) -  \Kop(V)\right) + \Kop(V).
\end{align*}
\end{proof}

\begin{corollary}\label{decompositionwithLiederivative}
The following decomposition holds 
\begin{equation*}
T_{\varphi}\Omega^3_{+}(M)  \approx \left( \{ (fg,0), f\in C^\infty(M)\} + \mathrm{Im} \Lop^* \right)\oplus (\ker \Lop\cap \ker \tr)
\end{equation*}
\end{corollary}

\section{The second variation of $\tilde {\mathcal F}$} \label{sec:2var}

In this section we compute the second variation of the functional $ \mathcal{\tilde{F} } $ at static $G_2$-structures that is either torsion free or nearly-$G_2$.   Our main result is the following theorem.

\begin{theorem}\label{thm:second_variation}
Let $M$ be a $7$-manifold admitting a $G_2$-structure $\varphi$ which induces the Riemannian metric $g$ on $M$ and the dual $4$-form $\psi$, and suppose that $\varphi$ is either torsion free or nearly-$G_2$, namely its torsion satisfies $T=c g$, for some constant $c$. Consider a $2$-parameter family $\{\varphi_{t,s}\}_{t,s\in(-\epsilon,\epsilon)}$ of $G_2$-structures, such that $\varphi_{(0,0)}=\varphi$, the induced metrics $g_{t,s}$ have unit volume and
\begin{align*}
\left.\frac{\partial}{\partial t}\right|_{t=s=0} \varphi_{t,s}&= h\diamond \varphi +X\lrcorner\psi,\\
\left.\frac{\partial}{\partial s}\right|_{t=s=0} \varphi_{t,s}&= w\diamond\varphi + Y\lrcorner \psi,
\end{align*}
for some symmetric $2$-tensors $h,w$ and $1$-forms $X,Y$ on $M$. Then
\begin{equation*}
\begin{aligned}
D^2_\varphi \tilde{\mathcal F}((h,X), (w,Y)):=\left. \frac{\partial^2}{\partial t\partial s}\right|_{t=s=0} \tilde{\mathcal F}(\varphi_{t,s}) &= \int_M \langle K_1(h,X),w\rangle + \langle K_2(h,X),Y\rangle d\mu_g,
\end{aligned}
\end{equation*}
where, denoting by $\Delta_L h_{ij} =\Delta h_{ij} +2R_{aijb} h_{ab} - R_{ia}h_{aj} - R_{ja} h_{ai}=\Delta h_{ij} +2R_{aijb} h_{ab} - 12c^2h_{ij} $ the Lichnerowicz Laplacian, 
\begin{align*}
K_1(h,X)&=\Delta_L h -\frac{2}{3}\mathcal L_{\tilde B(h,X)} g   +\frac{1}{3}\left(- \Delta \tr h +\Div(\Div(h)) \right) g
\\
&+\frac{5c}{6}\mathcal L_X g + 7c \curl(h)_\sym-2c\Div(X) g +5c^2 h,\\
K_2(h,X)&=\Delta X +\frac{2}{3} \curl(\tilde B(h,X)) - 3 c\Div h +2c \nabla \tr h +\frac{11c}{3} \curl(X) +12 c^2 X.
\end{align*}
Moreover, $\varphi$ is a saddle point of $\tilde{\mathcal F}$ and it is a local minimum of the restriction of $\tilde{\mathcal F}$ in the conformal class $[\varphi]=\{f^3\varphi, f\in C^\infty(M), f>0\}$ of $\varphi$.
\end{theorem}

\subsection{Variation of the second order operators associated to $\tilde{\mathcal F}$}

\begin{proposition}\label{prop:variation_P}
Let $M$ be a $7$-manifold with a $G_2$-structure $\varphi$ inducing the Riemannian metric $g$, and let $\psi=*_g \varphi$. Let $\{\varphi_t\}_{t\in (-\varepsilon,\varepsilon)}$ be a smooth $1$-parameter family of $G_2$-structures inducing the Riemannian metric $g_t$ on $M$ such that $\varphi_0=\varphi$ and
    $$\left. \frac{d}{dt}\right|_{t=0} \varphi_t = h\diamond \varphi + X\lrcorner \psi$$
    for some symmetric $2$-tensor $h$ and $1$-form $X$ on $M$. Then, if $T_\varphi=c g$,

\begin{align*}
\left.\frac{\partial}{\partial t}\right|_{t=0}\hat P_1 &= \Delta_L h - \frac{2}{3} \mathcal L_{\tilde B(h,X)} g + \frac{5c}{6} \mathcal L_X g+c(\Div(X) - c\tr h) g + 7c^2 h +7 c \curl(h)_\sym, \\
\left.\frac{\partial}{\partial t}\right|_{t=0} \tilde P_1&=\Delta_L h -\frac{2}{3}\mathcal L_{\tilde B(h,X)} g+\frac{5c}{6}\mathcal L_X g +\frac{35 c^2}{3} h + 7c \curl(h)_\sym+\frac{5c}{3} \left( \Div(X) - c\tr h \right)g,\\
\left.\frac{\partial}{\partial t}\right|_{t=0} P_1 &=\Delta_L h -\frac{2}{3}\mathcal L_{\tilde B(h,X)} g   +\frac{1}{3}\left( -\Delta \tr h +\Div(\Div(h)) \right) g
\\
&+\frac{5c}{6}\mathcal L_X g + 7c \curl(h)_\sym-2c\Div(X) g\\
\left.\frac{\partial}{\partial t}\right|_{t=0} P_2 &=\Delta X +\frac{2}{3} \curl(\tilde B(h,X)) - 3c \Div h +2c \nabla \tr h +\frac{11c}{3} \curl(X) +12 c^2 X.
\end{align*}
\end{proposition}

\begin{proof}
The proposition is a direct consequence of Lemma \ref{lemma:2_order} and Lemma \ref{lemma:low_order} that follow.

First, using Lemma \ref{lemma:2_order} and Lemma \ref{lemma:low_order} we obtain
\begin{align*}
&\left. \frac{\partial}{\partial t} \right|_{t=0} \left( -\ric -\frac{1}{3} \mathcal L_{VT} g - \frac{2}{3} (T\circ (VT\lrcorner \phi))_\sym +\tr T T_\sym \right) =\\
&=\Delta_L h - \frac{2}{3} \mathcal L_{\tilde B(h,X)} g - \frac{8c}{3} \mathcal L_X g+c(\Div(X) - c\tr h) g + 7c^2 h +7 c \curl(h)_\sym +\frac{7c}{2} \mathcal L_X g,\\
&=\Delta_L h - \frac{2}{3} \mathcal L_{\tilde B(h,X)} g + \frac{5c}{6} \mathcal L_X g+c(\Div(X) - c\tr h) g + 7c^2 h +7 c \curl(h)_\sym,
\end{align*}
which proves the variation formula of $\hat P_1$.

Then, since 
\begin{align*}
\tilde P_1&= \hat P_1 +\left(\frac{1}{3} |T|^2 -\frac{1}{9} |\Vop T|^2\right) g = \hat P_1 +\left(\frac{1}{6}\scal -\frac{1}{6}(\tr T)^2 -\frac{1}{3} |T|^2\right)g,\\
P_2&=\frac{2}{3}\Div T +\frac{1}{3} \nabla \tr T +\frac{1}{3}\tr T \Vop T,
\end{align*} 
using Lemma \ref{lemma:low_order} once more proves the result.
\end{proof}

\subsubsection{Variation formulae}
\begin{lemma}[Variation of second order terms]\label{lemma:2_order}
    Let $M$ be a $7$-manifold with a $G_2$-structure $\varphi$ inducing the Riemannian metric $g$, and let $\psi=*_g \varphi$. Let $\{\varphi_t\}_{t\in (-\varepsilon,\varepsilon)}$ be a smooth $1$-parameter family of $G_2$-structures such that $\varphi_0=\varphi$ and
    $$\left. \frac{d}{dt}\right|_{t=0} \varphi_t = h\diamond \varphi + X\lrcorner \psi$$
    for some symmetric $2$-tensor $h$ and $1$-form $X$ on $M$. Then, if $T_\varphi=c g$,
    \begin{align*}
        \left.\frac{\partial}{\partial t}\right|_{t=0} \scal&= 2(-\Delta \tr h + \Div(\Div h) -6c^2 \tr h),\\
  \left.\frac{\partial}{\partial t}\right|_{t=0} \ric &= -\Delta_L h +\mathcal L_{\Div h -\frac{1}{2}\nabla \tr h} g.
  \end{align*}

Moreover,
\begin{align*}
        \left.\frac{\partial}{\partial t}\right|_{t=0} \mathcal L_{\Vop T} g &=  \mathcal L_{-\Div h + \nabla \tr h +\curl(X) +6c X} g,\\
        \left.\frac{\partial}{\partial t}\right|_{t=0} \Div T &= \Delta X + \curl(\Div h)-c\Div h -c \curl(X),\\
        \left.\frac{\partial}{\partial t}\right|_{t=0} \nabla \tr T &= \nabla \Div X -c \nabla \tr h,
    \end{align*}
and
    \begin{equation}\label{eqn:symmetric_D}
        \left.\frac{\partial}{\partial t}\right|_{t=0} \left(-\ric -\frac{1}{3}\mathcal L_{\Vop T} g\right)= \Delta_L h-\frac{2}{3}\mathcal L_{\tilde B(h,X)} g- \frac{8c}{3}\mathcal L_X g,
        \end{equation}
        \begin{equation}\label{eqn:vector_D}
        \begin{aligned}
        &\left.\frac{\partial}{\partial t}\right|_{t=0} \left(\frac{2}{3}\Div T+ \frac{1}{3} \nabla\tr T\right)=\\
        &= \Delta X_k +\frac{2}{3} \curl(\tilde B(h,X))_k -\frac{2c}{3} \Div h+\frac{4c}{3}\curl(X) -\frac{c}{3}\nabla \tr h  - 2c^2 X.    \end{aligned}
        \end{equation}
\end{lemma}

\begin{proof}
Recall that we assume that the variation satisfies $T=cg$ at $t=0$, which implies that, at $t=0$, $\Vop T=0$ and $\ric = 6c^2 g$.

The variation of the Ricci and scalar curvature operators is standard in the literature, see for instance \cite{besse}.

To compute the variation of $\mathcal L_{\Vop T} g$ we first compute the variation of $\Vop T$. Using \eqref{eqn:torsion_variation} the contraction identities we obtain
\begin{equation}\label{eqn:VT_evol}
\begin{aligned}
\left.\frac{\partial}{\partial t}\right|_{t=0} \Vop T_j &= \left.\frac{\partial}{\partial t}\right|_{t=0} (T_{ab}\varphi_{cdj}g^{ac}g^{bd}), \\
&= ( cX_l \varphi_{lab} + \nabla_k h_{la}\varphi_{klb} +\nabla_a X_b)\varphi_{jab},   \\
&=6c X_j +\nabla_k h_{la} (g_{kj} g_{la} - g_{ka} g_{lj} - \psi_{klja}) + \curl(X)_j, \\
&=\nabla_j \tr h - \Div h_j + \curl(X)_j +6c X_j.\end{aligned}
\end{equation}

Since $\mathcal L_{\Vop T} g_{ij} = \nabla_i \Vop T_j +\nabla_j \Vop T_i$, the result follows.

Therefore,
\begin{align*}
\left.\frac{\partial}{\partial t}\right|_{t=0} \left(-\ric -\frac{1}{3}\mathcal L_{\Vop T} g\right) &=\Delta_L h -\mathcal L_{\Div h -\frac{1}{2}\nabla\tr h} g - \frac{1}{3} \left( \mathcal L_{ - \Div h+\curl(X)+\nabla \tr h + 6c X} g \right), \\
&=\Delta_L h-\frac{2}{3}\mathcal L_{\Div h -\frac{1}{4}\nabla\tr h +\frac{1}{2}\curl(X) + 3 cX} g,\\
&=\Delta_L h-\frac{2}{3}\mathcal L_{\tilde B(h,X)} g- \frac{8c}{3}\mathcal L_X g,
\end{align*}
where $\tilde B(h,X)=\Div h -\frac{1}{4}\nabla\tr h +\frac{1}{2}\curl(X) - cX$, which proves \eqref{eqn:symmetric_D}.

We now compute the variation of the operators $\Div T$ and $\nabla \tr T$.
\begin{align*}
&\left. \frac{\partial}{\partial t}\right|_{t=0} \Div T_k=\left. \frac{\partial}{\partial t}\right|_{t=0}  \left(\nabla_a T_{bk} g^{ab}\right), \\
&=  \nabla_a \left(\left. \frac{\partial}{\partial t}\right|_{t=0} T_{ak} \right) - c(\nabla_a h_{ak} +\nabla_a h_{ak} -\nabla_k h_{aa}    ) - c(\nabla_a h_{ka}  +\nabla_kh_{aa}  -\nabla_a h_{ak}),\\
&= \nabla_a (T_{ap}h_{pk}+ T_{ap}X_m \varphi_{mpk} + \nabla_p h_{qa}\varphi_{pqk} + \nabla_a X_k ) -2c \Div h_k,\\
&=\nabla_a (ch_{ak}+ cX_m \varphi_{mak} + \nabla_p h_{qa}\varphi_{pqk} + \nabla_a X_k ) -2c \Div h_k\\
&=-c \Div h_k - c \curl(X)_k +\nabla_a \nabla_p h_{qa} \varphi_{pqk} +\Delta X_k\\
&=-c \Div h_k - c \curl(X)_k +\left(\nabla_p \nabla_a h_{qa} + R_{apmq} h_{ma} +R_{apma} h_{qm}\right) \varphi_{pqk} +\Delta X_k,\\
&=-c \Div(h)_k - c \curl(X)_k +\curl(\Div h)_k +\Delta X_k, 
\end{align*}
where we have used that for nearly $G_2$ structures the Ricci tensor equals to $\ric = 6c^2 g$ as well as the fact that $ R_{apmq} h_{ma}$ is symmetric in $p$ and $q$. Now
\begin{align*}
\left. \frac{\partial}{\partial t}\right|_{t=0} \nabla_k\tr T &=\left. \frac{\partial}{\partial t}\right|_{t=0}  \nabla_k (T_{ab}g^{ab}) \\ 
&= \nabla_k \left( \left.\frac{\partial}{\partial t}\right|_{t=0} (T_{ab} g^{ab}) \right) \\
&=\nabla_k \left( (T_{am} h_{mb} +T_{am} X_l \phi_{lmb} +\nabla_m h_{la} \phi_{mlb} +\nabla_a X_b )g^{ab}\right) -2\nabla_k (T_{ab} h_{ab}),\\
&=\nabla_k ( ch_{ab} + c X_l \phi_{lab} +\nabla_m h_{la}\phi_{mlb} +\nabla_a X_b) g^{ab} -2c\nabla_k\tr h,\\
&=-c\nabla_k \tr h +\nabla_k \Div(X).
\end{align*}

It follows that
\begin{equation}\label{eqn:vector_evol}
\begin{aligned}
&\left.\frac{\partial}{\partial t}\right|_{t=0} \left(\frac{2}{3}\Div T+ \frac{1}{3} \nabla\tr T\right)=\\
&= \frac{2}{3}(\Delta X + \curl(\Div h)-c\Div h -c \curl(X))+\frac{1}{3}(\nabla \Div X -c \nabla \tr h).
\end{aligned}
\end{equation}

Next, we observe that we can express all the second order terms in \eqref{eqn:vector_evol}
 that are not $\Delta X$ in terms $\Delta X$, $\curl(\tilde B)$, and lower order terms, in the sense that
\begin{equation}\label{eqn:curl_tilde_B}
\frac{2}{3} \curl(\Div h) + \frac{1}{3} \nabla \Div X = \frac{2}{3} \curl(\tilde B(h,X)) +\frac{1}{3} \Delta X +2c \curl(X) - 2c^2 X.
\end{equation}
 To see this, we compute
\begin{align*}
   & \curl(\tilde B(h,X))_k = \\
   &=\curl(\Div h)_k +\frac{1}{2} \nabla_i \nabla_a X_b \varphi_{abj} \varphi_{kij}+\frac{c}{2}\nabla_a X_b \psi_{abij} \varphi_{kij} -c\curl(X)_k,\\
    &=\curl(\Div h)_k  +\frac{1}{2} \nabla_i \nabla_a X_b (g_{ak} g_{bi}-g_{ai}g_{bk} -\psi_{abki}) -2c \curl(X)_k-c\nabla_iX_j \varphi_{ijk},\\
    &=\curl(\Div h)_k  +\frac{1}{2} \nabla_i \nabla_k X_i -\frac{1}{2} \Delta X_k - \frac{1}{2}\nabla_i \nabla_a X_b \psi_{abki}-3c\curl(X)_k,\\
    &=\curl(\Div h)_k +\frac{1}{2}\nabla_i \nabla_k X_i - \frac{1}{2}\Delta X_k -\frac{1}{4}(\nabla_i\nabla_a X_b -\nabla_a\nabla_i X_b)\psi_{abki} - 3c \curl(X)_k,\\
    &=\curl(\Div h)_k +\frac{1}{2}\nabla_k \Div X +\frac{1}{2}R_{km}X_m  -\frac{1}{2}\Delta X_k -\frac{1}{4}R_{iamb}X_m \psi_{abki}-3c \curl(X)_k.
\end{align*}
Then, using that $\ric=6c^2 g$ and that the contraction $R_{iamb} \psi_{abki}=0$ due to the first Bianchi identity, we obtain \eqref{eqn:curl_tilde_B}.

Finally, substituting \eqref{eqn:curl_tilde_B} into \eqref{eqn:vector_evol} we obtain \eqref{eqn:vector_D}.
\end{proof}

\begin{lemma} [Variations of lower order terms]\label{lemma:low_order}
Let $M$ be a $7$-manifold with a $G_2$-structure $\varphi$ inducing the Riemannian metric $g$, and let $\psi=*_g \varphi$. Let $\{\varphi_t\}_{t\in (-\varepsilon,\varepsilon)}$ be a smooth $1$-parameter family of $G_2$-structures on $M$ such that $\varphi_0=\varphi$ and
    $$\left. \frac{d}{dt}\right|_{t=0} \varphi_t = h\diamond \varphi + X\lrcorner \psi$$
    for some symmetric $2$-tensor $h$ and $1$-form $X$ on $M$. Then, if $T_\varphi=c g$,
\begin{align*}
\left.\frac{\partial}{\partial t}\right|_{t=0} \tr T T_{\sym} &= c(\Div(X)-c\tr h) g + 7c^2 h + 7c \curl(h)_\sym +\frac{7c}{2}\mathcal L_X g,\\
\left.\frac{\partial}{\partial t}\right|_{t=0} (T\circ (\Vop T\lrcorner \varphi))_\sym &= 0,\\
\left.\frac{\partial}{\partial t}\right|_{t=0} \tr T \Vop T &= 42 c^2 X +7c\nabla \tr h -7c \Div(h) +7 c \curl(X),
\end{align*}
and
\begin{align*}
\left.\frac{\partial}{\partial t}\right|_{t=0} |\Vop T|^2 &= 0,\\
\left.\frac{\partial}{\partial t}\right|_{t=0} |T|^2 &=2 c \Div(X) - 2c^2 \tr h,\\
\left.\frac{\partial}{\partial t}\right|_{t=0} (\tr T)^2 &= 14 c \Div(X) - 14 c^2 \tr h.
\end{align*}
\end{lemma}
\begin{proof}
By \eqref{eqn:torsion_variation} we obtain
\begin{align*}
\left.\frac{\partial}{\partial t}\right|_{t=0} \tr T &=T_{ma}h_{am} +T_{ma}X_b \varphi_{bam} + \nabla_a h_{bm}\varphi_{abm} +\nabla_m X_m - 2T_{ab}h_{ab},\\
&=\Div X -c\tr h.
\end{align*}
Therefore,
\begin{align*}
\left.\frac{\partial}{\partial t}\right|_{t=0} \tr T T_{ij} &= c(\Div X - c \tr h)g_{ij}+7c(T_{ia}h_{aj} +T_{ia}X_b \varphi_{baj} + \nabla_a h_{bi}\varphi_{abj} +\nabla_i X_j)\\
&=c( \Div(X)-c\tr h)g_{ij} +7c (c h_{ij} + c X_b \varphi_{bij} +\nabla_a h_{bi}\varphi_{abj} + \nabla_i X_j),
\end{align*}
thus
\begin{align*}
\left.\frac{\partial}{\partial t}\right|_{t=0} \tr T T_{\sym} &= c(\Div(X)-c\tr h) g + 7c^2 h + 7c \curl(h)_\sym +\frac{7c}{2}\mathcal L_X g.
\end{align*}
Similarly, by \eqref{eqn:VT_evol} and since $\Vop T=0$ at $t=0$,
\begin{equation*}
\left.\frac{\partial}{\partial t}\right|_{t=0} \tr T \Vop T_k = 42 c X_k+7c \nabla_k \tr h - 7c \Div h_k +7c \curl(X)_k ),
\end{equation*}
moreover
\begin{equation*}
\left.\frac{\partial}{\partial t}\right|_{t=0} (\tr T)^2 = 14c (\Div X-c\tr h).
\end{equation*}
Now, again since $\Vop T=0$ at $t=0$ we obtain
\begin{align*}
\left.\frac{\partial}{\partial t}\right|_{t=0} |\Vop T|^2 &= 0,\\
 \left.\frac{\partial}{\partial t}\right|_{t=0} \left(T_{ia}\Vop T_b \varphi_{cdj} g^{ad}g^{bc} \right)&=c\left(\left.\frac{\partial}{\partial t}\right|_{t=0} \Vop T_b\right) \varphi_{bij},
\end{align*}
therefore, 
\begin{equation*}
\left.\frac{\partial}{\partial t}\right|_{t=0} (T\circ (\Vop T\lrcorner \varphi))_\sym = 0.
\end{equation*}
\end{proof}

\subsection{Proof of Theorem \ref{thm:second_variation} }

Let $\varphi$ be a $G_2$-structure inducing the Riemannian metric $g$ with $\vol(g)=1$ on a $7$-manifold $M$ which is a critical point of $\tilde{\mathcal F}$. Suppose that $\{\varphi_{t,s}\}_{t,s\in (-\epsilon,\epsilon)}$ be a smooth $2$-parameter family of $G_2$-structures on $M$ with
\begin{align*}
\left.\frac{\partial}{\partial t}\right|_{(0,0)} \varphi_{t,s} &= h\diamond \varphi + X\lrcorner \psi,\\
\left.\frac{\partial}{\partial s}\right|_{(0,0)} \varphi_{t,s} &= w\diamond \varphi + Y\lrcorner \psi,\\
\varphi_{0,0}&=\varphi,\\
\vol(g_{t,s})&=1.
\end{align*}
By Proposition \ref{prop:norm_F_variation} we know that
\begin{equation*}
\frac{\partial}{\partial s} \tilde{\mathcal F}(\varphi_{t,s}) 
= \int_M \langle P_1(\varphi_{t,s})-\frac{5}{7} \mathcal F(\varphi_{t,s}) g_{t,s}, w\rangle_{t,s} +\langle P_2(\varphi_{t,s}),Y\rangle_{t,s} d\mu_{g_{t,s}}.
\end{equation*}
Since $\varphi$ is a critical point of $\tilde{\mathcal F}$, we know by Proposition \ref{prop:norm_F_variation} that $P_1(\varphi)=\frac{5}{7} \mathcal F(\varphi) g$ and $P_2(\varphi)=0$, hence
\begin{align*}
\left.\frac{\partial^2}{\partial t\partial s}\right|_{t=s=0} \tilde{\mathcal F}(\varphi_{t,s})= \int_M \langle K_1(h,X), w \rangle + \langle K_2(h,X), Y\rangle d\mu_g
\end{align*}
where
\begin{align*}
K_1(h,X)&= \left.\frac{d}{dt}\right|_{t=0} (P_1(\varphi_{t,0}) -\frac{5}{7} \mathcal F(\varphi_{t,0}) g_{t,0})\\
K_2(h,X)&=\left.\frac{d}{dt}\right|_{t=0} P_2(\varphi_{t,0}).
\end{align*}
Since $(t,s)\mapsto \tilde{\mathcal F}(\varphi_{t,s})$ is a smooth function if follows that 
$$\left.\frac{\partial^2}{\partial t\partial s}\right|_{t=s=0} \tilde{\mathcal F}(\varphi_{t,s}) = \left.\frac{\partial^2}{\partial s\partial t}\right|_{t=s=0} \tilde{\mathcal F}(\varphi_{t,s})$$ 
hence the operator $(h,X) \mapsto (K_1(h,X),K_2(h,X))$ is symmetric.

Finally, by Proposition \ref{prop:variation_P}, we know that if in addition $\varphi$ satisfies $T=cg$ then
\begin{equation}\label{eqn:K12}
\begin{aligned}
K_1(h,X)&=\Delta_L h -\frac{2}{3}\mathcal L_{\tilde B(h,X)} g   +\frac{1}{3}\left(- \Delta \tr h +\Div(\Div(h)) \right) g
\\
&+\frac{5c}{6}\mathcal L_X g + 7c \curl(h)_\sym-2c\Div(X) g +5c^2 h,\\
K_2(h,X)&=\Delta X +\frac{2}{3} \curl(\tilde B(h,X)) - 3 c\Div h +2c \nabla \tr h +\frac{11c}{3} \curl(X) +12 c^2 X.
\end{aligned}
\end{equation}
The remaining assertions of Theorem \ref{thm:second_variation} follow from the more detailed analysis of the behaviour of the operators $K_1$, $K_2$ into the conformal and $G_2$-transverse-traceless directions, in Proposition \ref{prop:conformal_2var} and Proposition \ref{prop:tt_2var} respectively, below.

\subsection{Conformal deformations of torsion free or nearly $G_2$-structures} 

In this section, we focus on the behaviour of the normalized $G_2$-Hilbert functional within the conformal class of a given $G_2$-structure which is either torsion free or nearly-$G_2$. 

Our first result asserts that, in this case, the second variation of $\tilde{\mathcal F}$ is positive definite, hence $\varphi$ is a local minimum of $\tilde{\mathcal F}$ in its conformal class.

\begin{proposition}\label{prop:conformal_2var}
Suppose that $\varphi$ is a $G_2$-structure on $M$ inducing the unit-volume Riemannian metric $g$ and the dual $4$-form $\psi$, and its torsion satisfies $T=cg$ for some constant $c$, and consider a conformal variation $\{\varphi_t\}_{t\in (-\epsilon,\epsilon)}$ of $\varphi$, inducing the family $\{g_t\}_{t\in(-\epsilon,\epsilon)}$ of unit-volume Riemannian metrics on $M$, such that $\varphi_0=\varphi$ and
$$\left.\frac{\partial}{\partial t}\right|_{t=0} \varphi_t = h\diamond\varphi,$$ where $h=fg$ for some smooth function $f$ with $\int_M f d\mu=0$. We then have
\begin{equation}\label{eqn:conformal_2nd}
\begin{aligned}
\left.\frac{\partial}{\partial t}\right|_{t=0}( P_1(\varphi_t)-\frac{5}{7}\mathcal F(\varphi_t)g_t) 
&=(-\Delta f +5c^2 f )g+\frac{1}{2}\mathcal L_{\nabla f} g, \\
\left.\frac{\partial}{\partial t}\right|_{t=0} P_2&=11 c \nabla f.
\end{aligned}
\end{equation}
In particular, the projection of the operator  $f\mapsto(K_1(f g,0),K_2(f g,0))$ on the conformal directions is described by the operator
\begin{equation}\label{eqn:conformal_operator}
f\mapsto -6\Delta f +35c^2 f,
\end{equation}
which has strictly positive eigenvalues. In particular, the second variation of the normalized $G_2$-Hilbert functional satisfies
$$D^2\tilde{\mathcal F}((fg,0),(fg,0)) \geq 35c^2 \int_M f^2 d\mu_g>0.$$
\end{proposition}

\begin{proof}
Setting $h=fg$ and $X=0$ in \eqref{eqn:K12}, and using 
\begin{align*}
\tilde B(h,X)&= \Div(h) -\frac{1}{4}\nabla \tr h +\frac{1}{2}\curl(X) - cX=-\frac{3}{4}\nabla f, \\
\Delta_L h_{jk}&=\Delta h_{jk} +2R_{ajkb}h_{ab} -R_{ja}h_{ak}-R_{ka}h_{aj}=\Delta f g_{jk}, \\
\Delta \tr h - \Div(\Div (h))&=6 \Delta f\\
(\curl(h)_{\sym})_{jk}&=\frac{1}{2}(\nabla_a h_{bj}\phi_{abk}+\nabla_a h_{bk}\phi_{abj})=\frac{1}{2}(\nabla_a f \phi_{ajk}+\nabla_a f \phi_{akj})=0,\\
\nabla \tr h - \Div(h)&=6\nabla f,
\end{align*}
we obtain that
\begin{align*}
\left.\frac{\partial}{\partial t}\right|_{t=0} P_1&= \Delta f g+\frac{1}{2}\mathcal L_{\nabla f} g -2\Delta f g=-\Delta f  g+\frac{1}{2}\mathcal L_{\nabla f} g,\\
\left.\frac{\partial}{\partial t}\right|_{t=0} P_2&=-3c\nabla f+14c\nabla f=11c\nabla f.
\end{align*}
Equation \eqref{eqn:conformal_2nd} follows from the fact that
\begin{align*}
\left.\frac{\partial}{\partial t}\right|_{t=0} \mathcal F(\varphi_t) g_t &= \left(\left.\frac{\partial}{\partial t}\right|_{t=0} \mathcal F(\varphi_t) \right)g + 2\mathcal F(\varphi)h,\\
&=\left(\int_M \langle h,P_1(\varphi)\rangle d\mu_g -7c^2 f\right) g,\\
&=\left(-\frac{5c^2}{2}\int_M  7f d\mu_g -7c^2 f\right) g\\
&=-7c^2 f g.
\end{align*}
Finally, if $\eta$ is any smooth function on $M$,
\begin{align*}
\int_M \left(\langle K_1(fg,0), \eta g\rangle +\langle K_2(fg,0),0\rangle \right)d\mu_g &= \int_M \langle (-\Delta f +5c^2 f)g +\frac{1}{2}\mathcal L_{\nabla f} g, \eta g\rangle d\mu_g\\
&=\int_M (-6\Delta f + 35c^2 f)\eta d\mu_g,
\end{align*}
proves \eqref{eqn:conformal_operator}.
\end{proof}

Given a $G_2$-structure $\varphi$ inducing a Riemannian metric $g$ on $M$, let $[\varphi]$ denote the conformal class of $\varphi$, namely
$$[\varphi]=\{\tilde\varphi=e^{3f} \varphi, f\in C^\infty(M)\}.$$
Recall from \cite{flows2} that $\tilde\varphi=e^{3f}\varphi$ induces the Riemannian metric $\tilde g=e^{2f} g$ on $M$.

By Proposition \ref{prop:conformal_2var}, if $\varphi$ is a torsion-free or nearly-$G_2$ then it is a local minimum of the restriction $\tilde{\mathcal{F}}:[\varphi]\rightarrow \mathbb R$ of $\tilde{\mathcal F}$ in $[\varphi]$. 

\begin{proposition}\label{prop:F_conformal}
In the setting described above, if $\varphi$ is a torsion-free $G_2$-structure, it is the unique minimum $\mathcal F$ (and $\tilde{\mathcal F}$) in its conformal class $[\varphi]$.
\end{proposition}

To analyze in more detail the behaviour of $\tilde{\mathcal F}$ within a given conformal class $[\varphi]$, recall the following result from \cite{flows2}.
\begin{proposition}[Proposition 2.156 in \cite{flows2}]\label{prop:conformal_change}
Given a $G_2$-structure $\varphi$ and any $f\in C^\infty(M)$, let $\tilde \varphi=e^{3f}\varphi$. Then the torsion $\tilde T$ of $\tilde \varphi$ is given by
$$\tilde T_{pq}=e^f (T_{pq}+\nabla_m f \varphi_{mpq}).$$
\end{proposition}
A direct corollary of Proposition \ref{prop:conformal_change} is the following.
\begin{corollary}
Under the assumptions of Proposition \ref{prop:conformal_change}, the following identities hold.
\begin{align*}
|\tilde T|_{\tilde g}^2 &= e^{-2f} \left(|T|_g^2+2\langle \Vop T,\nabla f\rangle_g + 6|\nabla f|_g^2 \right),\\
\tilde \Vop \tilde T &= \Vop T+6\nabla f,\\
|\tilde\Vop \tilde T|_{\tilde g}^2 &= e^{-2f} \left( |\Vop T|^2 +12\langle \Vop T,\nabla f\rangle_g +36 |\nabla f|_g^2\right). 
\end{align*}
\end{corollary}

\begin{proof}[Proof of Proposition \ref{prop:F_conformal}]
Using \eqref{eqn:F_second_form}, we immediately see that $\mathcal F(\varphi)$ can be expressed in terms of the data $\varphi$, $g$ and $f$ as
\begin{equation*}
\begin{aligned}
\mathcal F(\tilde \varphi)=\mathcal F(f) &= \int_M \left(3|\nabla f|^2+\langle \Vop T,\nabla f\rangle -\frac{1}{2}|T|^2 +\frac{1}{6} |\Vop T|^2 \right) e^{5f} d\mu,\\
&=\int_M \left[3|\nabla f|^2 +\left(-\frac{1}{5} \Div\Vop T-\frac{1}{2}|T|^2 +\frac{1}{6} |\Vop T|^2 \right) \right] e^{5f}d\mu,\\
&=\int_M \left(3|\nabla f|^2 +\frac{1}{5} \tr P_1 \right)e^{5f}d\mu,
\end{aligned}
\end{equation*}
where the second equation is obtained via integration by parts and the third equation follows from \eqref{eqn:trP1_equiv}.

 Setting $v^{4/5}=e^{2f}$, we have that $e^{5f}= v^2$, hence $\mathcal F$ can be expressed in terms of $v$ as
 \begin{equation*}
\mathcal F(v) = \int_M \left( \frac{12}{25} |\nabla v|^2+ \frac{1}{5}\tr P_1 \right) v^2 d\mu_g.
\end{equation*}

If $\varphi$ is a torsion-free $G_2$-structure, then clearly $\tr P_1=0$, and thus $\mathcal F(v)>0$ and $\tilde{\mathcal F}(v)>0$, unless $v$ is constant.

\end{proof}

\subsection{$G_2$-transverse-traceless deformations of torsion free or nearly $G_2$-structures.}
We now turn our attention to $G_2$-transverse-traceless deformations of torsion-free and nearly $G_2$-structures.

\begin{proposition} \label{prop:tt_2var}
Suppose that $\varphi$ is a $G_2$-structure on $M$ inducing the Riemannian metric $g$ and the dual $4$-form $\psi$, and its torsion satisfies $T=cg$ for some constant $c$. Consider a variation $\{\varphi_t\}_{t\in(-\epsilon,\epsilon)}$ of $G_2$-structures inducing the family $\{g_t\}_{t\in(-\epsilon,\epsilon)}$ of Riemannian metrics on $M$ such that $\varphi_0=\varphi$ and 
$$\left.\frac{\partial}{\partial t}\right|_{t=0} \varphi_t = h\diamond\varphi +X\lrcorner \psi,$$ 
where $(h,X)$ is a $G_2$-transverse-traceless deformation of $\varphi$, namely $(h,X)\in \ker \Lop \cap \ker \tr$. Then
\begin{align*}
K_1(h,X)=\left. \frac{\partial}{\partial t}\right|_{t=0} \left(P_1(\varphi_t)-\frac{5}{7}\mathcal F(\varphi_t)g_t \right) &= \Delta_L h +\frac{5c}{6}\mathcal L_X g +7c \curl(h)_\sym  -\frac{5c}{3} \Div(X) g +5c^2 h,\\
K_2(h,X)=\left. \frac{\partial}{\partial t}\right|_{t=0} P_2(\varphi_t) &=\Delta X +\frac{31c}{2}\curl(X) +9 c^2 X.
\end{align*}

The projection $\Pi(K_1(h,X), K_2(h,X))$ onto the space of $G_2$-transverse-traceless deformations of $\varphi$ is given by the  symmetric operator 
\begin{align*}
\mathcal B(h,X)= \Pi \left( \Delta_L h + 7c \curl(h)_\sym, \Delta X +\frac{49c}{3} \curl(X) +\frac{22c^2}{3}X \right),
\end{align*}
which has infinitely many negative eigenvalues.
\end{proposition}
\begin{proof}
Consider a $G_2$-transverse-traceless variation $(h,X)$ of $\varphi$, namely assume that 
\begin{align*}
\Lop(h,X)&=\Div h + \frac{1}{2}\curl  X -cX=0\\
\tr h &=0.
\end{align*}
 It follows that $\tilde B(h,X) = \Lop(h,X) =0$. Therefore, $\Div(h)=-\frac{1}{2}\curl(X) + c X$, hence
\begin{align*}
\Div(\Div(h))=&\nabla_a \left(-\frac{1}{2}\curl(X)_a+cX_a\right),\\
&=-\frac{1}{2}\nabla_a(\nabla_b X_c \varphi_{bca}) +c\Div(X),\\
&=-\frac{1}{2} \nabla_a\nabla_b X_c \varphi_{bca} -\frac{1}{2} \nabla_b X_c T_{bm}\psi_{mbca} + c\Div(X),\\
&=c\Div(X).
\end{align*}

Thus,
\begin{align*}
\left.\frac{\partial}{\partial t}\right|_{t=0} P_1 &=\Delta_L h -\frac{2}{3}\mathcal L_{\tilde B(h,X)} g   +\frac{1}{3}\left( -\Delta \tr h +\Div(\Div(h)) \right) g
\\
&+\frac{5c}{6}\mathcal L_X g + 7c \curl(h)_\sym-2c\Div(X) g,\\
&=\Delta_L h  +\frac{c}{3} \Div(X) g +\frac{5c}{6}\mathcal L_X g +7c \curl(h)_\sym -2c\Div(X) g,\\
&=\Delta_L h +\frac{5c}{6}\mathcal L_X g +7c \curl(h)_\sym  -\frac{5c}{3} \Div(X) g,
\end{align*}
and
\begin{align*}
\left.\frac{\partial}{\partial t}\right|_{t=0} P_2 &=\Delta X_k +\frac{2}{3} \curl(\tilde B(h,X))_k - 3 c\Div h +2c \nabla \tr h +\frac{11c}{3} \curl(X) +12 c^2 X,\\
&=\Delta X -3c \Div h +\frac{11c}{3}\curl(X) +12 c^2 X,\\
&=\Delta X -3c \left( -\frac{1}{2}\curl(X) + c X \right)+\frac{11c}{3} \curl (X) +12 c^2 X,\\
&=\Delta X +\left( \frac{3c}{2}+\frac{11c}{3} \right) \curl(X) + 9 c^2 X,\\
&=\Delta X+\frac{31c}{6} \curl(X) +9 c^2 X.
\end{align*}
The expressions for $K_1$ and $K_2$ then follow from
\begin{align*}
\left.\frac{\partial}{\partial t}\right|_{t=0} \mathcal F(\varphi_t) g_t &= \left(\left.\frac{\partial}{\partial t}\right|_{t=0} \mathcal F(\varphi_t) \right)g + 2\mathcal F(\varphi)h,\\
&=\left(\int_M \langle h,P_1(\varphi)\rangle d\mu_g -7c^2 f\right) g,\\
&=\left(-\frac{5c^2}{2}\int_M   \tr hd\mu_g -7c^2 f\right) g\\
&=-7c^2 f g.
\end{align*}

Now, suppose further that $(w,Y)$ satisfies $\tr w=0$ and $L(w,Y)=0$, and $\Pi: \mathcal S^2\times \Omega^1 \rightarrow \ker \Lop \cap \ker \tr$ is the $L^2$-othrogonal projection.

We then have
\begin{align*}
&\int_M \langle \Pi (K_1(h,X),K_2(h,X)), (w,Y)\rangle d\mu=\\
&=\int_M \langle (K_1(h,X),K_2(h,X)), (w,Y)\rangle d\mu=\\
&=\int_M \langle \Delta_L h +7c \curl(h)_\sym +\frac{5c}{6}\mathcal L_X g   -\frac{5c}{3} \Div(X) g + 5c^2 h, w\rangle + \langle \Delta X +\frac{31c}{2}\curl(X) +9 c^2 X ,Y \rangle d\mu\\
&=\int_M \langle \Delta_L h +7c \curl(h)_\sym +\frac{5c}{3}\left(\frac{1}{2}\mathcal L_X g   - \frac{1}{7}\Div(X) g\right)+5c^2h, w\rangle d\mu\\
&+ \int_M\langle \Delta X +\left(\frac{31c}{2} +\frac{5c}{6}\right)\curl(X) +\left(9 c^2-\frac{5c^2}{3}\right) X +\frac{5c}{3}\left(-\frac{1}{2}\curl(X) + c X\right),Y \rangle d\mu\\
&=\int_M \langle \Delta_L h +7c \curl(h)_\sym+5c^2h,w\rangle + \langle \Delta X +\frac{49c}{3} \curl(X) +\frac{22c^2}{3} X,Y\rangle d\mu\\
&+\int_M \langle \Kop \left(-\frac{5c}{3} X\right), (w,Y)\rangle d\mu.
\end{align*}
Since $\mathrm{Im}\Kop$ is $\langle\cdot,\cdot\rangle_{L^2}$-orthogonal to $\ker L\cap \ker \tr$, and the operator $\curl$ acting on $2$-tensors and $1$-forms is symmetric when $T=cg$, we obtain that
\begin{align*}
&\int_M \langle \Pi (K_1(h,X),K_2(h,X)), (w,Y)\rangle d\mu=\\
&=\int_M \langle \Delta_L h +7c \curl(h)_\sym+5c^2 h,w\rangle + \langle \Delta X +\frac{49c}{3} \curl(X) +\frac{22c^2}{3} X,Y\rangle d\mu,\\
&=\int_M \langle h, \Delta_L w +7c \curl(w)_\sym +5c^2 w\rangle + \langle X, \Delta Y +\frac{49c}{3} \curl(Y) +\frac{22c^2}{3} Y\rangle d\mu,\\
&=\int_M \langle (h,X),\Pi( K_1(w,Y), K_2(w,Y) )d\mu.
\end{align*}
Thus 
$$\mathcal B(h,X):=\Pi\left(\Delta_L h + 7c \curl(h)_\sym+5c^2 h, \Delta Y+\frac{49c}{3} \curl(Y) +\frac{22c^2}{3} Y\right)=\Pi(K_1(h,X),K_2(h,X)),$$
and $\mathcal B$ is a symmetric operator.

Now, define the symmetric bilinear form 
\begin{align*}
&\mathcal E((h,X), (w,Y) ) =\\
&= \int_M \langle \nabla h,\nabla w\rangle -\langle \riem(h),w\rangle - 7c\langle \curl(h)_\sym, w\rangle d\mu\\
&+\int_M \langle \nabla X,\nabla Y\rangle -\frac{49c}{3} \langle\curl(X),Y\rangle - \frac{22c^2}{3} \langle X,Y\rangle  d\mu
\end{align*}
defined on the Hilbert space of $H^1$ $G_2$-transverse-traceless variations $(h,X)$ of $\varphi$.  Clearly, for smooth variations, integrating by parts we obtain that
$$\mathcal E((h,X),(w,Y)) = \int_M \langle \mathcal B(h,X), (w,Y)\rangle d\mu.$$

It easy to verify that $\mathcal E$ satisfies the energy estimates
\begin{equation*}
|\mathcal E((h,X), (w,Y) )| \leq C || (h,X) ||_{H^1} ||(w,Y)||_{H^1},
\end{equation*}
and 
\begin{equation*}
\beta ||(h,X)||_{H^1}^2 \leq \mathcal E((h,X),(h,X)) +\alpha ||(h,X)||_{L^2},
\end{equation*}
for appropriate constants $\alpha,\beta, C$. This suffices to conclude that $\mathcal B$ has real spectrum consisting of a sequence $\lambda_i\rightarrow -\infty$ such that $\mathcal B(h,X)=\lambda_i (h,X)$ has a non-trivial solution in $H^1$, see for instance \cite{evans}. In particular there are only finitely many negative eigenvalues.

Let 
$$\mathcal L(h,X)=\left(\Delta_L h+7c \curl(h)_\sym, \Delta X +\frac{49c}{3} \curl(X) +\frac{22c^2}{3} X\right).$$
By the proof of Proposition \ref{decompositionofG_2structures}
$$\Pi(h,X) =(h,X) - \Kop\circ (\Lop\circ \Kop)^{-1}(\Lop(h,X)).$$

Therefore, the eigenvalue equation 
$\mathcal B(h,X)=\Pi\circ \mathcal L(h,X)=\lambda(h,X)$ becomes
\begin{equation}\label{eqn:eigen}
\mathcal L(h,X) - \Kop\circ (\Lop\circ \Kop)^{-1}(\Lop(h,X)) = \lambda (h,X).
\end{equation}
Since $(h,X)\in H^1$, it follows $\Lop (h,X)\in L^2$. Therefore, since $\Lop\circ \Kop$ is elliptic, by Proposition \ref{decompositionofG_2structures},  elliptic regularity implies that
$$(\Lop\circ\Kop)^{-1}(\Lop(h,X)) \in H^2 \Longrightarrow \Kop\circ (\Lop\circ\Kop)^{-1}(\Lop(h,X)) \in H^1. $$
Since $\mathcal L$ is also elliptic, applying elliptic regularity in  \eqref{eqn:eigen} we obtain that $(h,X)\in H^3$. Bootstrapping the above argument, we eventually obtain that $(h,X)$ is $C^\infty$, and thus the eigenfunctions of $\mathcal B$, in the subspace of $G_2$-transverse-traceless deformations of $\varphi$ are smooth.

 Therefore, the second variation of $\tilde{\mathcal F}$ satisfies 
$$D^2\tilde{\mathcal F}( (h,X),(h,X)) =\int_M \langle\mathcal B(h,X),(h,X)\rangle d\mu \leq 0$$
for all $G_2$-transverse-traceless $(h,X)$ which are $L^2$ orthogonal to a finite dimensional subspace, which corresponds to the positive eigenvalues of $\mathcal B$.

\end{proof}

\subsubsection{Static deformations of torsion-free $G_2$-structures}

Let $\{\varphi_t\}_{t\in(-\epsilon,\epsilon)}$ be a one parameter family of unit volume static $G_2$-structures on a compact $7$-manifold $M$, inducing the Riemannian metrics $\{g_t\}_{t\in (-\epsilon,\epsilon)}$, namely 
\begin{align*}
P_1(\varphi_t)&=0,\\
P_2(\varphi_t)&=0,\\
\vol(g_t)&=1,
\end{align*} for every $t\in (-\epsilon,\epsilon)$, and suppose that $\varphi=\varphi_0$ is torsion free, $g=g_0$ and $\psi=*_g \varphi$.

Let $h$ and $X$ be such that 
$$\left.\frac{\partial \varphi_t}{\partial t} \right|_{t=0} = h\diamond\varphi+X\lrcorner \psi.$$
By Proposition \ref{prop:F_conformal}, $\varphi$ is the unique static $G_2$-structure in its conformal class, while the operators $P_1,P_2$ are invariant under diffeomorphisms. Hence the only non-trivial behaviour will observed in the $G_2$-transverse-traceless directions,
$$(h,X) \in \ker L \cap \ker \tr.$$
In particular $\tilde B(h,X)=0$. Therefore, by Proposition \ref{prop:variation_P} we obtain that
\begin{align*}
\left.\frac{\partial}{\partial t}\right|_{t=0} P_1 &= \Delta_L h =\Delta h +2\riem(h)=0,\\
\left.\frac{\partial}{\partial t}\right|_{t=0} P_2 &=\Delta X=0.
\end{align*}

\begin{definition}
Let $(h,X)\in \ker L\cap \ker \tr$.
\begin{enumerate}
\item We call $(h,X)$ an infinitesimal static deformation of the torsion-free $G_2$-structure $\varphi$ if
\begin{equation*}
\Delta h + 2\riem(h)=0,\quad  \Delta X=0,
\end{equation*}
where $\riem(h)_{ij}=R_{aijb}h_{ab}$. 
\item We call $(h,X)$ an infinitesimal torsion free deformation of the torsion-free $G_2$-structure $\varphi$ if
\begin{equation*}
D_\varphi T(h,X)= \curl(h) + \nabla X=0.
\end{equation*}
\end{enumerate}
\end{definition}

\begin{remark}\label{rmk:static_imp}
If $(h,X)$ is an infinitesimal static deformation of a torsion free $G_2$-structure $\varphi$, then since $M$ is compact, $\Delta X=0$ implies that $X$ is parallel. In particular, $\curl(X)=0$ and $\Lop(h,X)=0$ implies that $\Div h=0$.
\end{remark}

\begin{remark}\label{rmk:tfree_imp}
If $(h,X)$ is an infinitesimal torsion free deformation of a torsion-free $G_2$-structure $\varphi$, contracting $\curl(h)+\nabla X=0$ with $\varphi$ and applying a contraction identity we obtain
\begin{equation}\label{eqn:t_free_d}
\nabla_a h_{bi} \varphi_{abj}\varphi_{kij} + \curl(X)_k = \nabla_k \tr h -\Div h +\curl(X)_k=0.
\end{equation}
Now, since $\tr h=0$ and $\Lop(h,X)=0$ implies that $\Div h=-\frac{1}{2}\curl(X)$, \eqref{eqn:t_free_d} implies that $\curl(X)=0$ and $\Div(h)=0$.
\end{remark}

\begin{proposition}
A pair $(h,X)\in \ker L\cap \ker \tr$ is an infinitesimal static deformation if and only if it is an infinitesimal torsion free deformation.
\end{proposition}
\begin{proof}
Suppose that $\Div h=0$ and $w$ is any symmetric $2$-tensor. Then 
\begin{align*}
\langle \curl(h),\curl(w)\rangle&=\nabla_a h_{bi}\varphi_{abj}  \nabla_p w_{qi}\varphi_{pqj},\\
&=\nabla_a h_{bi} \nabla_p w_{qi} (g_{ap}g_{bq}-g_{aq} g_{bp}-\psi_{abpq}),\\
&=\langle\nabla h,\nabla w\rangle-\nabla_a h_{bi} \nabla_b w_{ai} -\nabla_a h_{bi}\nabla_p w_{qi}\psi_{abpq}. 
\end{align*}

Therefore, integrating by parts, we obtain 
\begin{align*}
&\int_M \langle\curl(h),\curl(w)\rangle  d\mu = \\
&=\int_M (-\Delta h_{ij} w_{ij}  +\nabla_b\nabla_a h_{bi} w_{ai} +\nabla_p \nabla_a h_{bi} w_{qi}\psi_{abpq})d\mu\\
&=\int_M \left(-\Delta h_{ij} w_{ij}  +(\nabla_a\nabla_b h_{bi} +R_{bamb} h_{mi} +R_{bami} h_{bm}) w_{ai} \right.\\
& \left.+\frac{1}{2}(R_{pamb} h_{mi} +R_{pami} h_{bm}) w_{qi}\psi_{abpq}\right)d\mu\\
&=\int_M (-\Delta h_{ij} h_{ij} +R_{bami} h_{bm} w_{ai} +\frac{1}{2} R_{mipa} h_{bm}w_{qi}\psi_{pabq} )d\mu,\\
&=\int_M (-\Delta h_{ij} - R_{aijb}h_{ab} -\frac{1}{2} R_{aikl}\psi_{kljb}h_{ab} )w_{ij} d\mu
\end{align*}

since $\nabla \psi=0$, $\ric=0$, $\Div h=0$ and $R_{pamb}\psi_{pabq}=0$ due to the first Bianchi identity.

Recall that the $\Omega^7$-projection of the Riemann tensor 
$$\pi_7(\riem)_{aijb}=\frac{1}{3}R_{aijb}-\frac{1}{6} R_{aikl}\psi_{kljb},$$
satisfies
$$\pi_7(\riem)_{aijb}\varphi_{jbm} =\nabla_a T_{im} -\nabla_i T_{am} - T_{ap}T_{iq} \varphi_{pqm}=0,$$
since $\varphi$ is torsion free, by \cite{flows1}. It follows that
$\frac{1}{2} R_{aikl}\psi_{kljb} = R_{aijb}$, hence
\begin{equation}\label{eqn:curlhw}
\int_M \langle\curl(h),\curl(w)\rangle d\mu = \int_M (-\Delta h_{ij}-2R_{aijb}h_{ab}) w_{ij} d\mu =\int_M -\Delta_L h_{ij} w_{ij} d\mu. 
\end{equation}
If $(h,X)$ is an infinitesimal static deformation, then by Remark \ref{rmk:static_imp} we know that $\Div h=0$ and $\nabla X=0$. Therefore, setting $h=w$ in \eqref{eqn:curlhw} gives $\curl(h)=0$, so $(h,X)$ is an infinitesimal torsion free deformation.

On the other hand, if $(h,X)\in \ker L\cap \ker\tr$ is an infinitesimal torsion free deformation, we know from Remark \ref{rmk:tfree_imp} that $\curl(X)=0$ and $\Div h=0$.

Therefore, then $\curl(h)+\nabla X=0$. Therefore, $\Delta X = -\Div\curl(h)=0$, so $X$ is in fact parallel and $\curl(h)=0$.

Finally, by \eqref{eqn:curlhw}, for any $w$ we have
\begin{align*}
-\int_M \langle\Delta h+2\riem(h),w\rangle d\mu &= \int_M \langle \curl(h),\curl(w)\rangle d\mu =0,
\end{align*}
hence $\Delta h+2\riem(h)=0$. Thus, $(h,X)$ is an infinitesimal static deformation.

\end{proof}


\end{document}